\documentclass{article}
\usepackage{graphicx} 
\usepackage{amsmath} 
\usepackage{amsfonts}
\usepackage{amsthm}

\theoremstyle{definition}

\usepackage{amssymb}
\usepackage{wrapfig}
\usepackage{tikz}
\usepackage{environ}

\usepackage{lmodern}  
\usepackage[none]{hyphenat} 
\usepackage{graphicx} 
\graphicspath{ {./images/} } 
\usepackage{caption}
\usepackage{subcaption}
\usepackage{booktabs}
\usepackage{parskip} 
\usepackage[export]{adjustbox} 
\usepackage{float} 
\usepackage{comment}
\usepackage{mathptmx}
\usepackage{bbm}
\usepackage[pdfencoding=auto, psdextra]{hyperref}

\usepackage{fullpage}
\usepackage{authblk}

\usepackage{titlesec}

\usepackage{appendix}
\usepackage[svgnames]{xcolor}
\colorlet{LightGray}{gray!20}
\usepackage{minted}
\numberwithin{equation}{section}
\usepackage{graphicx} 
\usepackage{lmodern}

\newtheorem{prop}{Proposition} 
\newtheorem{lem}[prop]{Lemma}
\newtheorem{thm}[prop]{Theorem}

\newtheorem{remark}[prop]{Remark}
\usepackage[
    natbib=true,
    style=numeric,
    sorting=none
]{biblatex}
\begin{document}

\title{The Burgers-FKPP advection-reaction-diffusion equation with cut-off}
\author[$\dagger$]{Nikola Popovi\'c}
\author[$\ddagger$]{Mariya Ptashnyk}
\author[$\dagger$]{Zak Sattar\thanks{z.sattar@sms.ed.ac.uk}}
\affil[$\dagger$]{School of Mathematics and Maxwell Institute for Mathematical Sciences, University of Edinburgh, James Clerk Maxwell Building, King’s Buildings, Peter Guthrie Tait Road, Edinburgh EH9 3FD, United Kingdom}
\affil[$\ddagger$]{School of Mathematical \& Computer Sciences and Maxwell Institute for Mathematical Sciences, Heriot-Watt University, Edinburgh EH14 4AP, United Kingdom}

\date{\today}

\maketitle


\begin{abstract}
\noindent
We investigate the effect of a Heaviside cut-off on the front propagation dynamics of the so-called Burgers-Fisher- Kolmogoroff-Petrowskii-Piscounov (Burgers-FKPP) advection-reaction-diffusion equation. We prove the existence and uniqueness of a ``critical" travelling front solution in the presence of a cut-off in the reaction kinetics and the advection term, and we derive the leading-order asymptotics for the speed of propagation of the front in dependence on the advection strength and the cut-off parameter. 
Our analysis relies on geometric techniques from dynamical systems theory and specifically, on geometric desingularisation, which is also known as ``blow-up''.

\end{abstract}


\section{Introduction} \label{intro}
Partial differential equations (PDEs) of reaction-diffusion type are frequently derived from discrete many-particle systems in the large-scale limit as the number $N$ of particles becomes infinite. However, discrepancies are observed between the propagation speeds of front solutions that are found numerically in the underlying many-particle systems and the corresponding speeds in the reaction-diffusion equations derived in the limit as $N\to \infty$ \cite{BREUER1994259, complaticegassmodel, PhysRevE.58.107}. To remedy these discrepancies, Brunet and Derrida \cite{PhysRevE.56.2597} introduced a cut-off in the resulting reaction kinetics; for a general reaction-diffusion equation with reaction kinetics $f(u)$, their modification takes the form
\begin{equation}
    \frac{\partial u}{\partial t}=\frac{\partial^2 u}{\partial x^2}+f(u)\psi(u,\varepsilon),
    \label{general}
\end{equation}
where the cut-off function $\psi$ {\it a~priori} only has to satisfy $\psi(u,\varepsilon)\equiv 1$ when $u> \varepsilon$ and $\psi(u,\varepsilon)<1$ for $u<\varepsilon$. \textcolor{black}{Here, $u=u(x,t)$, with $x\in\mathbb{R}$ and $t\geq 0$.} The motivation in \cite{PhysRevE.56.2597} was that, in $N$-particle systems, no reaction can take place if the particle density is below some threshold value $\frac{1}{N}=\varepsilon\ll 1$. Applying the method of matched asymptotics, they showed that for Fisher reaction kinetics $f(u)=u(1-u)$ \cite{Fisher} in \eqref{general} and a Heaviside cut-off $H(u-\varepsilon)$, with $H\equiv 0$ when $u<\varepsilon$, the shift in the propagation speed \textcolor{black}{$c$ of the front connecting the homogeneous rest states $u=1$ and $u=0$ that is }due to a cut-off is, to leading order, given by $\Delta c=2-c=\frac{\pi^2}{(\ln\varepsilon)^2}+\mathcal{O}[(\ln\varepsilon)^{-3}]$. 

The above asymptotics has been derived rigorously by Dumortier {\it et al.} \cite{Dumortier_2007} for a more general class of scalar reaction-diffusion equations and a broad family of cut-off functions. They applied geometric desingularisation, or ``blow-up'' \cite{dumortier1996canard,krupa33extending}, to construct propagating front solutions to Equation~\eqref{general} as heteroclinic connections in the corresponding first-order system of ordinary differential equations (ODEs) after transformation to a co-moving frame. Geometric desingularisation has since been applied successfully in the study of numerous other reaction-diffusion equations with a cut-off \cite{Dumortier_2007, popdeg, Popović_2011, POPOVIC20121976, article2}. A well-developed alternative approach for determining the leading-order shift in the propagation speed due to a cut-off relies on a variational principle \cite{PhysRevE.75.051106, PhysRevE.76.051101, article3, article4}.

\medskip

In this article, we extend the results of \cite{Dumortier_2007,popdeg, Popović_2011, POPOVIC20121976, article2} to a family of advection-reaction-diffusion equations of the form
\begin{equation}
    \begin{aligned}
        \frac{\partial u}{\partial t}+g(u)\frac{\partial u}{\partial x}=\frac{\partial^2 u}{\partial x^2}+f(u),
    \end{aligned}
    \label{general 2}
\end{equation}
with the advection term $g(u)\frac{\partial u}{\partial x}$ describing the directed transport of $u$. \textcolor{black}{As far as we are aware, the impact of a cut-off on front propagation in advection-reaction-diffusion equations of the type in \eqref{general 2} has not
been studied before. Specifically, we consider the following equation,
\begin{equation}
        \frac{\partial u}{\partial t}+ k u\frac{\partial u}{\partial x}=\frac{\partial^2 u}{\partial x^2}+u(1-u),
        \label{Burger-FKPP1}
\end{equation}
which is known as Burgers-Fisher-Kolmogorov-Petrowskii-Piscounov (Burgers-FKPP) equation \cite{Murray2002,ablowitz1987topics,alma998723063502466, 10.1201/9781420034967}; here, $k u$ is the transport velocity, with $k>0$ a real parameter. 
Equation~\eqref{Burger-FKPP1} is a special case of the generalised FKPP equation, see \cite{gfkpp}, which models competing genotypes in a population and which has found applications in a variety of fields that include fluid dynamics, population modelling, and chemical kinetics. It hence serves as a prototypical model that illustrates the interaction between advection, reaction, and diffusion mechanisms in more general advection-reaction-diffusion equations of the type in \eqref{general 2}. In particular, it realises a pushed front propagation regime which is not present in standard reaction-diffusion with Fisher reaction kinetics.}

To study travelling wave solutions to \eqref{Burger-FKPP1}, we introduce the travelling wave variable $\xi=x-ct$, where $c$ denotes the propagation speed. Setting $U(\xi)=u(x,t)$, we obtain the travelling wave equation
\begin{equation}
    -cU'+kUU'=U''+U(1-U),
    \label{secondorder}
\end{equation}
subject to the boundary conditions $U(-\infty)=1$ and $U(\infty)=0$. The corresponding solution defines a front for the Burgers-FKPP equation, \eqref{Burger-FKPP1}, which connects the two rest states $u=1$ and $u=0$.

Defining $V=U'$ in \eqref{secondorder}, we obtain the first-order system
\begin{equation}
    \begin{aligned}
        U'&=V, \\
        V'&=-cV+kUV-U(1-U),
        \label{firstorder}
    \end{aligned}
\end{equation}
which has equilibria at $Q^-:=(1,0)$ and $Q^+:=(0,0)$. Clearly, heteroclinic orbits for \eqref{firstorder} and front solutions to \eqref{Burger-FKPP1} are equivalent.
The following result can be found in \cite{2021ZaMP...72..163M}, where the existence of travelling front solutions to \eqref{Burger-FKPP1} is shown rigorously.
\begin{thm}\cite[Theorem~4.1]{2021ZaMP...72..163M} 
     Equation~\eqref{firstorder} admits a heteroclinic connection between $Q^-$ and $Q^+$ for $c\geq c_{\rm crit}$, where 
    \begin{equation}
        c_{\rm crit}=\begin{cases}
    2 &\text{if}\:\: k\leq 2,\\
    \frac{k}{2}+\frac{2}{k} &\text{if}\:\: k>2.
    \label{c(0)}
    \end{cases}
    \end{equation}
    Moreover, the corresponding front solution to \eqref{Burger-FKPP1} is pulled when $k\leq 2$ and pushed when $k>2$.
    For $k\geq 2$ and $c_{\rm crit}=\frac{k}{2}+\frac{2}{k}$, the heteroclinic orbit for \eqref{firstorder} is given explicitly by $V(U)=-\frac{k}{2}U(1-U).$
    \label{basic thm}
\end{thm}

\textcolor{black}{As is the case for standard reaction-diffusion,  Equation~\eqref{general}, advection should give no contribution when $u<\varepsilon\big(=\frac{1}{N}\big)$, where $N$ is the total number of particles in the underlying many-particle system \cite{burgeparticle}. Hence, it seems plausible that a cut-off should multiply both the reaction kinetics $f(u)$ and the advection term $g(u)\frac{\partial u}{\partial x}$ in \eqref{general 2}. 
Our aim in this article is hence to prove an analogue of Theorem~\ref{basic thm} for the Burgers-FKPP equation with a Heaviside cut-off $H(u-\varepsilon)$,
\begin{equation}
        \frac{\partial u}{\partial t}+ku\frac{\partial u}{\partial x} H(u-\varepsilon)=\frac{\partial^2 u}{\partial x^2}+u(1-u)H(u-\varepsilon), 
        \label{PDEcut1}
\end{equation}
where $k>0$ and $\varepsilon>0$ is the cut-off parameter, as before. Our focus on the Burgers-FKPP equation is further motivated by the fact that, even in the simple Burgers-type advection-diffusion equation \cite{Musha_1978}
\begin{equation}
     \frac{\partial u}{\partial t}+ku\frac{\partial u}{\partial x}=\frac{\partial^2 u}{\partial x^2},
     \label{simpleburger}
\end{equation}
a cut-off impacts on front propagation and the corresponding speed: it is well-known that Equation~\eqref{simpleburger} admits a travelling front solution connecting the rest states $u=1$ and $u=0$ which propagates with the unique speed $c=\frac{k}{2}$. One can then show that introduction of a Heaviside cut-off function $H(u-\varepsilon)$ in \eqref{simpleburger}, whence
\begin{equation}
    \frac{\partial u}{\partial t}+ku\frac{\partial u}{\partial x}H(u-\varepsilon)=\frac{\partial^2 u}{\partial x^2},
    \label{simpleburger with cut}
\end{equation}
induces the shift $\Delta c=\frac{k}{2}\varepsilon^2$ in the front propagation speed which can be derived either by matched asymptotics or via an adaptation of the approach developed in this article.}

Our main result can be formulated as follows. 
\begin{thm}\label{Myfirstthm}
    Let $\varepsilon\in[0, \varepsilon_0)$, with $\varepsilon_0>0$ sufficiently small, and let $k>0$. Then, there exists a unique, $k$-dependent propagation speed $c(\varepsilon)$ such that Equation~\eqref{PDEcut1} admits a unique critical front solution connecting the rest states $u=1$ and $u=0$. Moreover, $c(\varepsilon)=c(0)-\Delta c(\varepsilon)$, where $c(0)=\lim_{\varepsilon\to 0^+}c(\varepsilon)=c_{\rm crit}$ is the critical speed in the absence of a cut-off, see Theorem~\ref{basic thm}, with 
    \begin{equation}
        \Delta c(\varepsilon)=\begin{cases}
    \frac{\pi^2}{(\ln{\varepsilon})^2} &\text{if}\:\: \textcolor{black}{k< 2,}\\
    \textcolor{black}{\frac{\pi^2}{4(\ln{\varepsilon})^2} }&\textcolor{black}{\text{if}\:\: k=2,}\\
    \frac{2}{k^{1+8/k^2}}\frac{(k^2-4)^{4/k^2}}{\Gamma(1+4/k^2)\Gamma(1-4/k^2)}
    \varepsilon^{1-4/k^2} &\text{if}\:\: k>2,
    \end{cases}
    \label{dc in thm}
    \end{equation}  
    to leading order in $\varepsilon$.
\end{thm}
Here and in the following, $\Gamma(\cdot)$ denotes the standard Gamma function \cite[Section 6.1]{abramowitz+stegun}.

In particular, Theorem~\ref{Myfirstthm} hence implies that the front propagation speed in \eqref{Burger-FKPP1} is reduced by inclusion of a (Heaviside) cut-off.
\begin{remark}
    The above result is similar to that for the Nagumo equation with cut-off obtained in \cite{Popović_2011}, which also realises pulled and pushed front propagation regimes in dependence on a control parameter. Correspondingly, the correction to the front propagation speed found in the pulled regime is again of the order $\mathcal{O}[(\ln{\varepsilon})^{-2}]$, whereas in the pushed regime, it is proportional to a fractional power of $\varepsilon$. 
\end{remark}
\textcolor{black}{
\begin{remark}
While our choice of Heaviside cut-off in Equation~\eqref{PDEcut1} is mostly made for analytical tractability, we indicate in Section~\ref{discussionsection} below how Theorem~\ref{Myfirstthm} can be extended to more general choices of cut-off function $\psi(u,\varepsilon)$.
\end{remark}
}

\textcolor{black}{\begin{remark}
    We note that \eqref{PDEcut1} does not conserve mass due to the reaction kinetics $u(1-u)H(u-\varepsilon)$. However, the advection and diffusion terms can be written in mass conservation form as follows,
    \begin{equation*}
        \frac{\partial u}{\partial t}+\frac{\partial}{\partial x}F_{\varepsilon}(u)=u(1-u)H(u-\varepsilon),
    \end{equation*}
    where 
    \begin{align*}
        F_{\varepsilon}(u)&=k\int_{0}^u\sigma H(\sigma-\varepsilon)d\sigma-\frac{\partial u}{\partial x}
        =\begin{cases}
    \frac{k}{2}(u^2-\varepsilon^2)-\frac{\partial u}{\partial x}&\text{if}\:\: u>\varepsilon,\\
    -\frac{\partial u}{\partial x} &\text{if}\:\: u< \varepsilon.
    \end{cases} 
    \end{align*}
\end{remark}}

The article is organised as follows: in Section~\ref{section 2}, we apply geometric desingularisation (blow-up) to construct a singular heteroclinic orbit $\Gamma$ for Equation~\eqref{PDEcut1}. In Section~\ref{persistencesection}, we show that $\Gamma$ persists for $\varepsilon$ sufficiently small, therefore establishing Theorem~\ref{Myfirstthm}, and we provide numerical verification of our results. Finally, in Section~\ref{discussionsection}, we discuss our findings and outline future related work.

\section{Geometric desingularisation}
\label{section 2}
Introducing the travelling wave variable $\xi=x-ct$ and writing $u(x,t)=U(\xi)$ in \eqref{PDEcut1}, we obtain the system of equations
\begin{equation}
    \begin{aligned}
        U'&=V,\\
        V'&=-cV+kUV H(U-\varepsilon)-U(1-U)H(U-\varepsilon),\\
        \varepsilon'&=0
        \label{main}
    \end{aligned}
\end{equation}
in analogy to \eqref{firstorder}, where the $(U,V)$-subsystem has been extended by the trivial equation for the cut-off parameter $\varepsilon$.
We introduce the following blow-up transformation (geometric desingularisation) at the origin in \eqref{main}, which serves to desingularise the non-smooth transition between the outer and inner regions in $\{U=\varepsilon\}$:
\begin{equation}
    U=\bar{r}\bar{u},\quad V=\bar{r}\bar{v},\quad\text{and}\quad\varepsilon=\bar{r}\bar{\varepsilon}.
    \label{burger-blowup1}
\end{equation}
Here, $(\bar{u}, \bar{v}, \bar{\varepsilon})\in\mathbb{S}_+^2:=\{(\bar{u}, \bar{v}, \bar{\varepsilon})\ |\ \bar{u}^2+\bar{v}^2+\bar{\varepsilon}^2=1\}\cap\{\bar{\varepsilon}\geq 0\}$, with $\bar{r}\in[0, r_0]$ for $r_0>0$ sufficiently small.

As in \cite{Dumortier_2007,krupa33extending,Popović_2011,krupa2001blowup}, we will analyse \eqref{main} in two coordinate charts, $K_1$ and $K_2$, which are obtained by setting $\bar{u}=1$ and $\bar{\varepsilon}=1$ in \eqref{burger-blowup1}, respectively. The rescaling chart $K_2$ will cover the ``inner region" where $U<\varepsilon$, while the phase-directional chart $K_1$ will allow us to describe the dynamics in the ``outer" region, with $U>\varepsilon$. The transition between the two regions, at $\{U=\varepsilon\}$, will be realised in the overlap domain between these coordinate charts. We will construct the singular (in $\varepsilon$) heteroclinic orbit $\Gamma$ for \eqref{main} by combining appropriate portions thereof in the two charts. As will become apparent, the uniqueness of the ``critical" propagating front in Theorem~\ref{Myfirstthm} is a consequence of the fact that a unique choice of $c$ yields a persistent heteroclinic connection between $Q^-=(1,0)$ and $Q^+=(0,0)$. 

\textcolor{black}{
\begin{remark}
For any object $\Box$ in $(U,V,\varepsilon)$-space, we denote the corresponding blown-up object by $\overline{\Box}$. Moreover, in chart $K_i$, with $i=1,2$, that object will be denoted by $\Box_i$.
\end{remark}
}

\subsection{Dynamics in chart \texorpdfstring{$K_2$}{} (``Inner region")}
\label{section 2.1}
In this subsection, we construct the portion $\Gamma_{2}$ of the singular heteroclinic orbit $\Gamma$ in chart $K_2$. Setting $\bar{\varepsilon}=1$ in \eqref{burger-blowup1}, we have the transformation 
\begin{equation}
    U=r_2u_2,\quad V=r_2v_2,\quad\text{and}\quad\varepsilon=r_2,
    \label{K_2_blowup}
\end{equation}
which we apply to~\eqref{main} to obtain the system of equations
\begin{equation}
    \begin{aligned}
        u_2'&=v_2, \\
        v_2'&=-cv_2+kr_2u_2v_2H(u_2-1)-u_2(1-r_2u_2)H(u_2-1), \\
        r_2'&=0.
        \label{K_2before}
    \end{aligned}
\end{equation}
We consider \eqref{K_2before} in the inner region where $U<\varepsilon$, which is equivalent to $u_2<1$, by \eqref{K_2_blowup}. Therefore, $H(u_2-1)\equiv 0$, which implies that \eqref{K_2before} reduces to
\begin{equation}
    \begin{aligned}
        u_2'&=v_2,\\
        v_2'&=-cv_2,\\
        r_2'&=0.
        \label{K_2-H}
    \end{aligned}
\end{equation}

We define the line of equilibria $\ell_2^+=\{(0,0,r_2)\ |\ r_2\in[0,r_0]\}$ for \eqref{K_2-H}. Here, we are particularly interested in the point $Q_2^+=(0,0,0)\in\ell_2^+$, which is found by taking the singular limit as $r_2\to 0^+$ on $\ell_2^+$.
(While Equation~\eqref{K_2-H} admits equilibria for any $u_2\in(0,1)$, we only consider $u_2=0$, which corresponds to the point $Q^+$ before blow-up.) The eigenvalues of the linearisation of \eqref{K_2-H} about $Q_2^+$ are $-c$ and $0$ (double), where the second zero eigenvalue is due to the trivial $r_2$-equation. Taking $r_2\to 0^+$ in \eqref{K_2-H}, we find the system
\begin{equation}
    \begin{aligned}
         u_2'&=v_2,\\
        v_2'&=-c(0) v_2,
    \end{aligned}
\end{equation}
where we have used that $c\to c(0) =c_{\rm crit}$ as $r_2(=\varepsilon)\to 0^+$ and $c_{\rm crit}$ is as defined in Theorem~\ref{basic thm}. 

We can solve for the stable manifold $W_2^{\rm s}(Q_2^+)$ of $Q_2^+$ by writing 
\begin{equation*}
    \frac{dv_2}{du_2}=-c(0),
\end{equation*}
which we can integrate under the condition that $v_2(0)=0$; the unique solution is given by 
\begin{equation*}
    \Gamma_2: \:\: v_2(u_2)=-c(0)u_2.
\end{equation*}

Therefore, invoking again Theorem~\ref{basic thm}, we can write
\begin{equation}
    \Gamma_2: \:\:v_2(u_2) =\begin{cases}
    -2u_2 &\text{if}\:\: k\leq 2,\\
    -\left(\frac{k}{2}+\frac{2}{k}\right)u_2 &\text{if}\:\: k>2.
    \end{cases} 
    \label{Gamma2}
\end{equation}
\textcolor{black}{We emphasise that $\Gamma_2$ is linear in $u_2$ for both $k\leq 2$ and $k>2$; while the corresponding slope is $k$-dependent in the latter case, by \eqref{c(0)}, that difference is immaterial for the dynamics.}

\begin{figure}
    \centering
    \includegraphics[scale = 0.9]{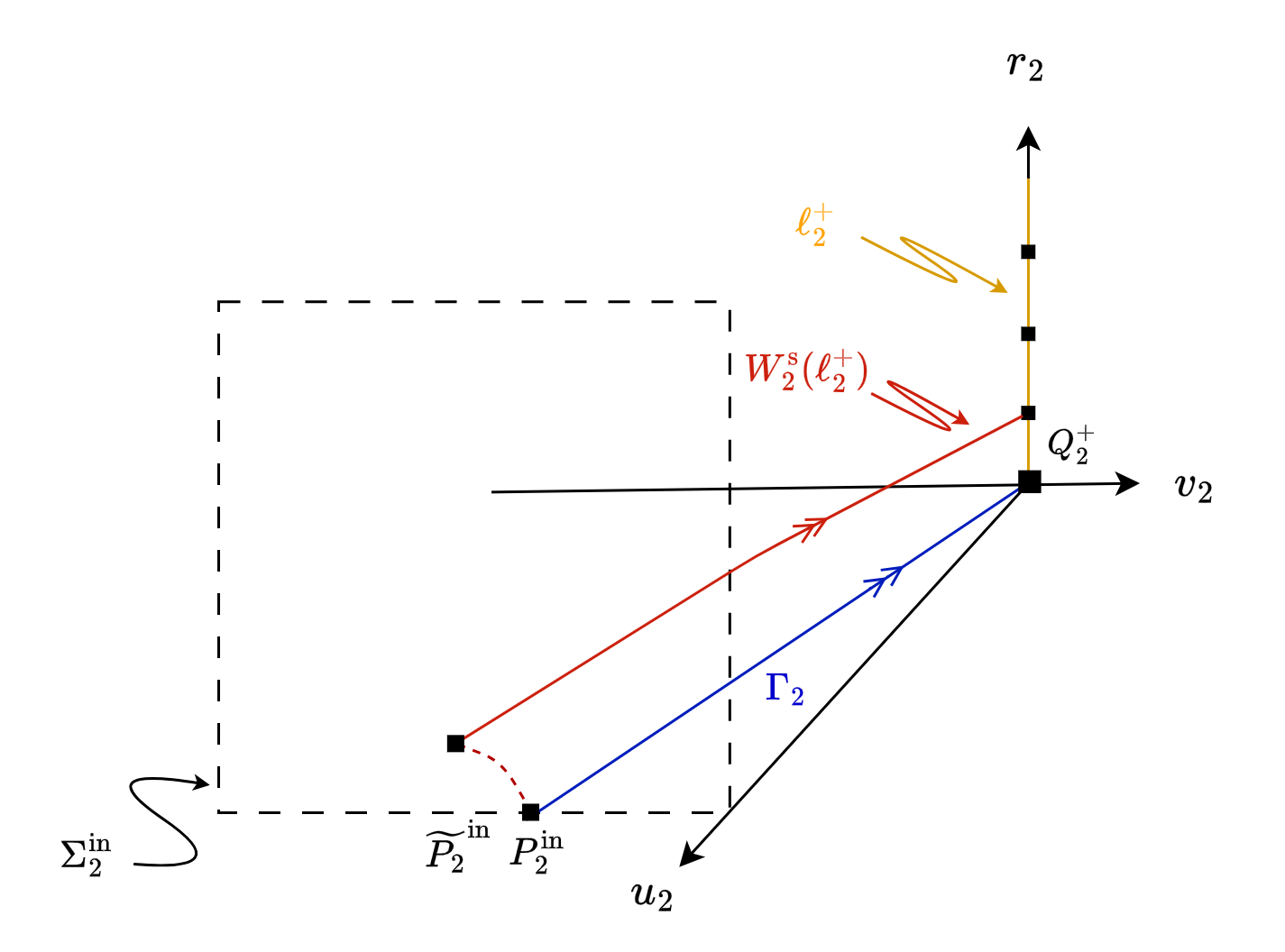}
    \caption{Geometry and dynamics in chart $K_2$.}
    \label{fig:K_2}
\end{figure}

Next, we introduce the entry section 
\begin{equation}\label{K2 in}
    \Sigma_2^{\rm in}=\{(1, v_2, r_2)\ |\ (v_2, r_2)\in[-v_0, 0]\times[0, r_0] \},
\end{equation}
which is the equivalent of the hyperplane $\{U=\varepsilon\}$ in chart $K_2$, to describe the transition of the orbit $\Gamma_2$ between the inner and outer regions; \textcolor{black}{here, $v_0>0$ is an appropriately chosen constant.} We define the entry point into $K_2$ as $P_2^{\rm in}=\Gamma_2\cap\Sigma_2^{\rm in}=(1,-c(0),0)$, where 
$P_2^{\rm in}=(1, -2, 0)$ for $k\leq 2$ and $P_2^{\rm in}=\big(1, -\big(\frac{k}{2}+\frac{2}{k}\big), 0\big)$ for $k>2$, by \eqref{Gamma2}. The geometry in chart $K_2$ is illustrated in Figure~\ref{fig:K_2}. \textcolor{black}{The singular orbit $\Gamma_2$ (in blue), which corresponds to the stable manifold $W_2^{\rm s}(Q_2^+)$ of the equilibrium at the origin, intersects the entry section $\Sigma_2^{\rm in}$ in $P_2^{\rm in}$. For $r_2\in(0,r_0]$, $\Gamma_2$ will perturb to the stable manifold $W_2^{\rm s}(\ell_2^+)$ (in red) of the line of equilibria $\ell_2^+$ (in yellow).}

\subsection{Dynamics in chart \texorpdfstring{$K_1$}{} (``Outer region")}
\label{section 2.2}
In this subsection, we will analyse the dynamics of Equation~\eqref{main} in the phase-directional chart $K_1$. Our aim is to construct the singular orbit $\Gamma_{1}$, which is the continuous extension of $\Gamma_{2}$, as defined in \eqref{Gamma2}, to $K_1$.

Setting $\bar{u}=1$ in \eqref{burger-blowup1}, we have 
\begin{equation}
    U=r_1,\quad V=r_1v_1,\quad\text{and}\quad\varepsilon=r_1\varepsilon_1,
    \label{K_1transformation}
\end{equation}
which we apply to~\eqref{main} in the outer region where $U>\varepsilon$ to find 
\begin{equation}
    \begin{aligned}
        r_1'&=r_1v_1,\\
        v_1'&=-cv_1+kr_1v_1H(1-\varepsilon_1)-(1-r_1)H(1-\varepsilon_1)-v_1^2,\\
        \varepsilon_1'&=-\varepsilon_1 v_1.
        \label{K_1}
    \end{aligned}
\end{equation}
Here, $H(1-\varepsilon_1)\equiv1$ due to $\varepsilon_1<1$ in chart $K_1$.
The system of equations in \eqref{K_1} has a line of equilibria at $\ell_{1}^-=\{(1,0,\varepsilon_1)\ |\ \varepsilon_1\in[0,\varepsilon_0]\}$ which corresponds to the steady state at $Q^{-}$ before blow-up. As other equilibria of \eqref{K_1} depend on $k$, we will discuss them systematically in Sections~\ref{section:k<2subsubsection} and~\ref{section:k>2subsubsection} below. The point $Q_1^-=(1,0,0)\in\ell_1^-$ is
obtained in the limit as $\varepsilon_1\to 0^+$. 

Since $\varepsilon=r_1\varepsilon_1$, we will have to consider both $r_1\to0$ and $\varepsilon_1\to0$ in the singular limit of $\varepsilon=0$. We will denote the corresponding portions of the singular orbit $\Gamma_1$ in the invariant planes $\{r_1=0\}$ and $\{\varepsilon_1=0\}$ by $\Gamma_{1}^+$ and $\Gamma_{1}^-$, respectively.

We introduce the exit section
\begin{equation}
    \Sigma_{1}^{\rm out}=\{(r_1,v_1, 1)\ |\ (r_1, v_1)\in[0, r_0]\times [-v_0, 0]\}
\end{equation}
to track $\Gamma_{1}^+$ as it leaves chart $K_1$; \textcolor{black}{here, $v_0>0$ is defined as in \eqref{K2 in}.} Clearly, $\Sigma_{1}^{\rm out}$ is equivalent to the entry section $\Sigma_2^{\rm in}$ in chart $K_2$ after transformation to $K_1$: as the change of coordinates $\kappa_{21}:K_2\to K_1$ between the two charts is given by
\begin{equation}
    \kappa_{21}:\ r_1=r_2u_2,\quad v_1=v_2u_2^{-1},\quad\text{and}\quad \varepsilon_1=u_2^{-1},
    \label{coordnatechange}
\end{equation}
we have $\kappa_{21}(\Sigma_2^{\rm in})=\Sigma_1^{\rm out}$. Correspondingly, we can write $P_{1}^{\rm out}=(0,-c(0),1)=\kappa_{21}(P_2^{\rm in})$ for the exit point in $\Sigma_1^{\rm out}$, where $P_{2}^{\rm in}=(1,-c(0), 0)$, as before.

\medskip
\textcolor{black}{In contrast to chart $K_2$, the singular geometry and dynamics in $K_1$ are qualitatively different for $k\leq 2$ and $k>2$, in that the corresponding phase portraits will not be topologically equivalent. Therefore, we consider these regimes in \eqref{main} separately.}

\subsubsection{Pulled front propagation: \texorpdfstring{$k\leq 2$}{}}
\label{section:k<2subsubsection}
We note that, when $k\leq 2$, the propagation speed $c$ reduces to $c(0)=2=c_{\rm crit}$ when either $r_1\to 0^+$ or $\varepsilon_1\to 0^+$, recall Theorem~\ref{basic thm}. In addition to the line of equilibria $\ell_1^-,$ we have an equilibrium at $P_1=(0,-1,0)$. A simple calculation shows the following result.

\begin{lem}
    The eigenvalues of the linearisation of \eqref{K_1} at $P_1$ are given by $-1$, $0$, and $1$, with corresponding eigenvectors $(1,k-1,0)^T, (0,1,0)^T$, and $(0,0,1)^T$, respectively.
\end{lem}

We first outline the construction of $\Gamma_1^+$. Taking $r_1\to 0^+$ in \eqref{K_1}, we obtain
\begin{equation}
    \begin{aligned}
        v_1'&=-2v_1-1-v_1^2,\\
        \varepsilon_1'&=-\varepsilon_1 v_1,
        \label{r_1=0}
    \end{aligned}
\end{equation}
which we can write as
\begin{equation}
    \frac{dv_1}{d\varepsilon_1}=\frac{(v_1+1)^2}{\varepsilon_1 v_1}.
    \label{eps_1 indep}
\end{equation}
To find a solution to Equation~\eqref{eps_1 indep} so that the orbit $\Gamma_1^+$ connects to $\Gamma_{2}$ in the section $\Sigma_1^{\rm out}=\kappa_{21}(\Sigma_2^{\rm in})$, we require $v_1(1)=-c(0)=-2(=v_2(1))$, by \eqref{coordnatechange}. The corresponding (unique) solution is given by 
\begin{equation}
    \Gamma_{1}^+:\ v_1(\varepsilon_1)=-\frac{1+W_0\big(\frac{e}{\varepsilon_1}\big)}{W_0\big(\frac{e}{\varepsilon_1}\big)},
    \label{gamma_1+pull}
\end{equation}
where $W_0$ denotes the Lambert W function \cite{Dence2013}, which is defined as the solution to $W_0(z) e^{W_0(z)}=z$. 
We note that $\Gamma_{1}^+\to P_1=(0,-1,0)$ as $\varepsilon_1\to 0^+$, which completes the construction.

To construct $\Gamma_1^-$, we take $\varepsilon_1\to 0^+$ in \eqref{K_1}, which yields 
\begin{equation}
    \begin{aligned}
        r_1'&=r_1v_1,\\
        v_1'&=-2v_1+kr_1v_1-(1-r_1)-v_1^2.
        \label{epslim}
    \end{aligned}
\end{equation}
Clearly, \eqref{epslim} is equivalent to the unmodified first-order system in \eqref{main} with $c=c(0)$ after blow-down, \textcolor{black}{i.e., after transformation to the original $(U,V,\varepsilon)$-space before the blow-up}: 
\begin{equation}
    \begin{aligned}
        U'&=V,\\
        V'&=-2V+kUV-U(1-U).
        \label{main_2}
    \end{aligned}
\end{equation}
While we cannot explicitly solve \eqref{main_2} for $k<2$, the following two results imply the existence of the orbit $\Gamma_{1}^-$. The first of these is obtained by simple linearisation. 
\begin{lem}
    The origin $Q^+$ in \eqref{main_2} is a degenerate stable node with eigenvalue $-1$ (double) and eigenvector $(-1,1)^T$, while the equilibrium at $Q^-=(1,0)$ is a saddle point with eigenvalues $\frac{k-2\pm\sqrt{k^2-4k+8}}{2}$ and corresponding eigenvectors $\Big(\frac{1}{2}\big(2-k\pm\sqrt{8-4k+k^2}\big), 1\Big)^T.$ 
    \label{Qpm lemma}
\end{lem}

Next, we show that \eqref{main_2} admits a trapping region for $k<2$; the proof is inspired by \cite[Theorem 2.1]{2021ZaMP...72..163M}.
\begin{prop}
    The curves $\{V=0\}$ and $\{V=-U(1-U)\}$ form a trapping region $\mathcal{T}$ for the flow of Equation~\eqref{main_2} when $k<2$. Moreover, the curve $\{V=-U(1-U)\}$ is invariant under the flow of \eqref{main_2} when $k=2$.
    \label{trappingregionprop}
\end{prop}

\begin{proof}
    Substitution of $V=0$ into \eqref{main_2} gives 
    \begin{equation}
        \begin{aligned}
            U'&=0, \\
            V'&=-U(1-U),
        \end{aligned}
    \end{equation}
    which implies $(0,1)\cdot (0, -U(1-U))^T=-U(1-U)<0$ due to $0<U<1.$
    Similarly, substituting $V=V(U)=-U(1-U)$ into \eqref{main_2}, we obtain
     \begin{equation}
        \begin{aligned}
            U'&=-U(1-U), \\
            V'&=U(1-U)(1-kU)
        \end{aligned}
    \end{equation}
    and $(-V'(U), 1)\cdot(U', V')^T=-(k-2)U^2(1-U)>0$ when $k<2$, which also implies that $V(U)=-U(1-U)$ is invariant under the flow of \eqref{main_2} when $k=2$.
\end{proof}
\begin{figure}[H]
    \centering
    \includegraphics[scale = 0.85]{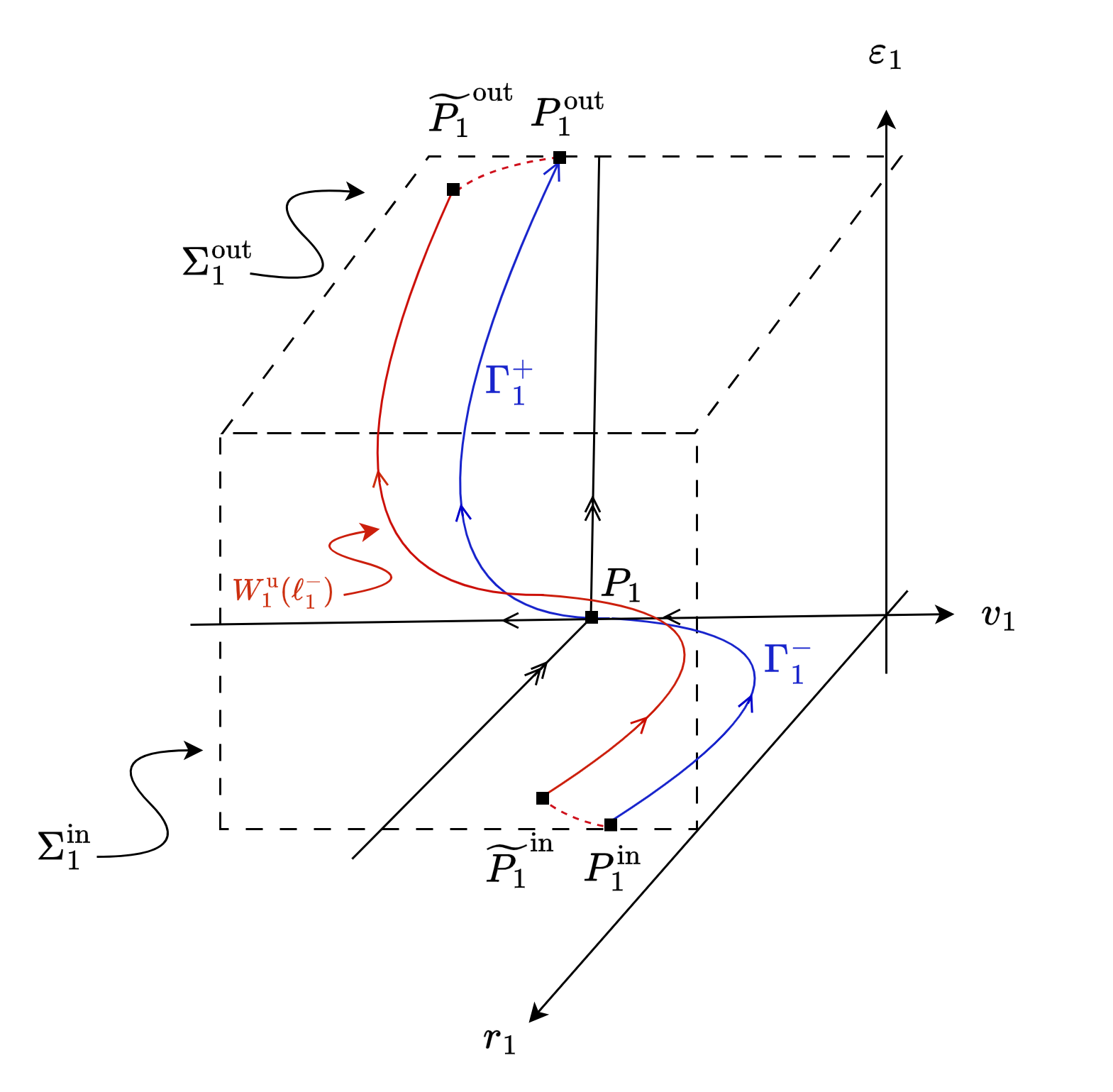}
    \caption{Geometry and dynamics in chart $K_1$ for $k\leq2$.}
    \label{fig:kleq 2}
\end{figure}

Since the trapping region $\mathcal{T}$ contains only the two equilibria $Q^-$ and $Q^+$, and since the divergence of the vector field in \eqref{main_2} is negative for $U\in[0,1)$ and all $V$, there are no periodic orbits in $\mathcal{T}$. Hence, there must exist a heteroclinic connection between $Q^-$ and $Q^+$. That connection must pass through the negative $V$-plane and is consistent with the stability properties of $Q^\mp$ stated in Lemma ~\ref{Qpm lemma}. It follows that the flow of \eqref{epslim} must enter an equivalent trapping region $\mathcal{T}_1$ in chart $K_1$. The closed region $\mathcal{T}_1$ is bounded by the lines $\{r_1=0\}$, $\{v_1=0\}$, and $\{v_1=r_1-1\}$ in the plane $\{\varepsilon_1=0\}$.  
Therefore, we can conclude that the orbit $\Gamma_{1}^-$ exists and is forward asymptotic to $P_1$. Defining the section 
\begin{equation}
    \Sigma_{1}^{\rm in}=\{(r_0,v_1, \varepsilon_1)\ |\ (v_1, \varepsilon_1)\in[-v_0, 0]\times [0, 1]\},
\end{equation}
\textcolor{black}{with $v_0>0$ as in \eqref{K2 in},} we see that the point of intersection $P_{1}^{\rm in}=\Gamma_{1}^-\cap \Sigma_{1}^{\rm in}$ is given by $P_{1}^{\rm in}=(r_0, v_{1}^{\rm in}, 0)$, where $v_{1}^{\rm in}>-1$, as $v_{1}^{\rm in}\in[r_0-1,0]$ by the proof of Proposition~\ref{trappingregionprop}.

Hence, the construction of $\Gamma_1=\Gamma_1^-\cup P_1\cup\Gamma_1^+$ is complete in the case where $k\leq 2$; see Figure~\ref{fig:kleq 2} for an illustration of the geometry in chart $K_1$ in that case. \textcolor{black}{The portions $\Gamma_1^-$ and $\Gamma_1^+$ (in blue) of $\Gamma_1$ are forward and backward asymptotic, respectively, to the equilibrium at $P_1$ and intersect the sections $\Sigma_1^{\rm in}$ and $\Sigma_1^{\rm out}$ in $P_1^{\rm in}$ and $P_1^{\rm out}$, respectively. For $\varepsilon_1\in(0,\varepsilon_0]$, $\Gamma_1$ will perturb to the unstable manifold $W_1^{\rm u}(\ell_1^-)$ (in red) of the line of equilibria $\ell_1^-$.}

\subsubsection{Pushed front propagation: \texorpdfstring{$k>2$}{}}
\label{section:k>2subsubsection}
We now consider the singular dynamics in chart $K_1$ in the pushed propagation regime where $k>2$. In analogy to the pulled regime, $\ell_1^-$ is still a line of equilibria for \eqref{K_1}. Since, however, $c\to c(0)=\frac{k}{2}+\frac{2}{k}$ as $r_1\to 0^+$, the point $P_1$ is no longer an equilibrium for \eqref{K_1}. Instead, we have two equilibria, 
at $\hat{P}_1=\big(0,-\frac{k}{2}, 0\big)$ and $\check{P}_1=\big(0,-\frac{2}{k}, 0\big)$, which undergo a saddle-node bifurcation as $k\to 2^+$. We are interested in the strong stable eigendirection of the linearisation about the origin in \eqref{firstorder}, i.e., in the absence of a cut-off. Since the heteroclinic orbit $V(U)=-\frac{k}{2}U(1-U)$ is the union of the unstable manifold $W^{\rm u}(Q^-)$ of $Q^-$ and the strong stable manifold $W^{\rm ss}(Q^+)$ of $Q^+$, we restrict our attention to $\hat{P}_1$. The point $\check{P}_1$ corresponds to the weak stable eigendirection at the origin in \eqref{firstorder}, which is not relevant here.

The following lemma summarises the stability properties of $\hat{P}_1$.
\begin{lem}
    The eigenvalues of the linearisation of \eqref{K_1} at $\hat{P}_1=\big(0,-\frac{k}{2}, 0\big)$ are given by $-\frac{k}{2}$, $\frac{k}{2}-\frac{2}{k}$, and $\frac{k}{2}$, with corresponding eigenvectors $\big(\frac{2}{k}, 1, 0\big)^T$, $(0,1,0)^T$, and $(0,0,1)^T$, respectively.
\end{lem}

We again first construct the portion $\Gamma_1^+$ of $\Gamma_1$. Taking $r_1\to 0^+$ in \eqref{K_1}, we obtain 
\begin{equation}
    \begin{aligned}
        v_1'&=-\Big(\frac{k}{2}+\frac{2}{k}\Big)v_1-1-v_1^2,\\
        \varepsilon_1'&=-\varepsilon_1 v_1,
        \label{r_1lim}
    \end{aligned}
\end{equation}
which we rewrite as 
\begin{equation*}
    \frac{dv_1}{d\varepsilon_1}=\frac{1+\big(\frac{k}{2}+\frac{2}{k}\big)v_1+v_1^2}{\varepsilon_1v_1}.
\end{equation*}
Solving by separation of variables, 
we find
\begin{equation*}
        \ln\frac{|k+2v_1|^{k^2}}{|kv_1+2|^4}=(\ln\varepsilon_1+\alpha)(k^2-4),
\end{equation*}
where $\alpha$ is a constant of integration. Exponentiating both sides in the above equation, we have 
\begin{equation}
    \begin{aligned}
        \frac{(k+2v_1)^{k^2}}{(kv_1+2)^4}&=\alpha'\varepsilon_1^{k^2-4},
        \label{gamma_1+push}
    \end{aligned}
\end{equation}
with $\alpha'=\pm e^{\alpha(k^2-4)}$. 

We choose $\alpha'$ so that the orbit $\Gamma_1^+$ connects to $\Gamma_{2}$ in the section $\Sigma_1^{\rm out}=\kappa_{21}(\Sigma_2^{\rm in})$. Thus, we require $v_1(1)=-c(0)=-\big(\frac{k}{2}+\frac{2}{k}\big)(=v_2(1))$, which is satisfied for $\textcolor{black}{\alpha'=\frac{(-4/k)^{k^2}}{(k^2/2)^4}}$. 

Finally, we note that $\Gamma_{1}^+$ is backward asymptotic to $\hat{P}_1$. Taking $\varepsilon_1\to 0^+$, we conclude that  
\begin{equation}
    \frac{(k+2v_1)^{k^2}}{(kv_1+2)^4}\to0
\end{equation}
must hold, which is only true when $v_1\to-\frac{k}{2}$. Hence, the construction of $\Gamma_1^+$ is complete.

\begin{figure}[H]
    \centering
    \includegraphics[scale = 0.85]{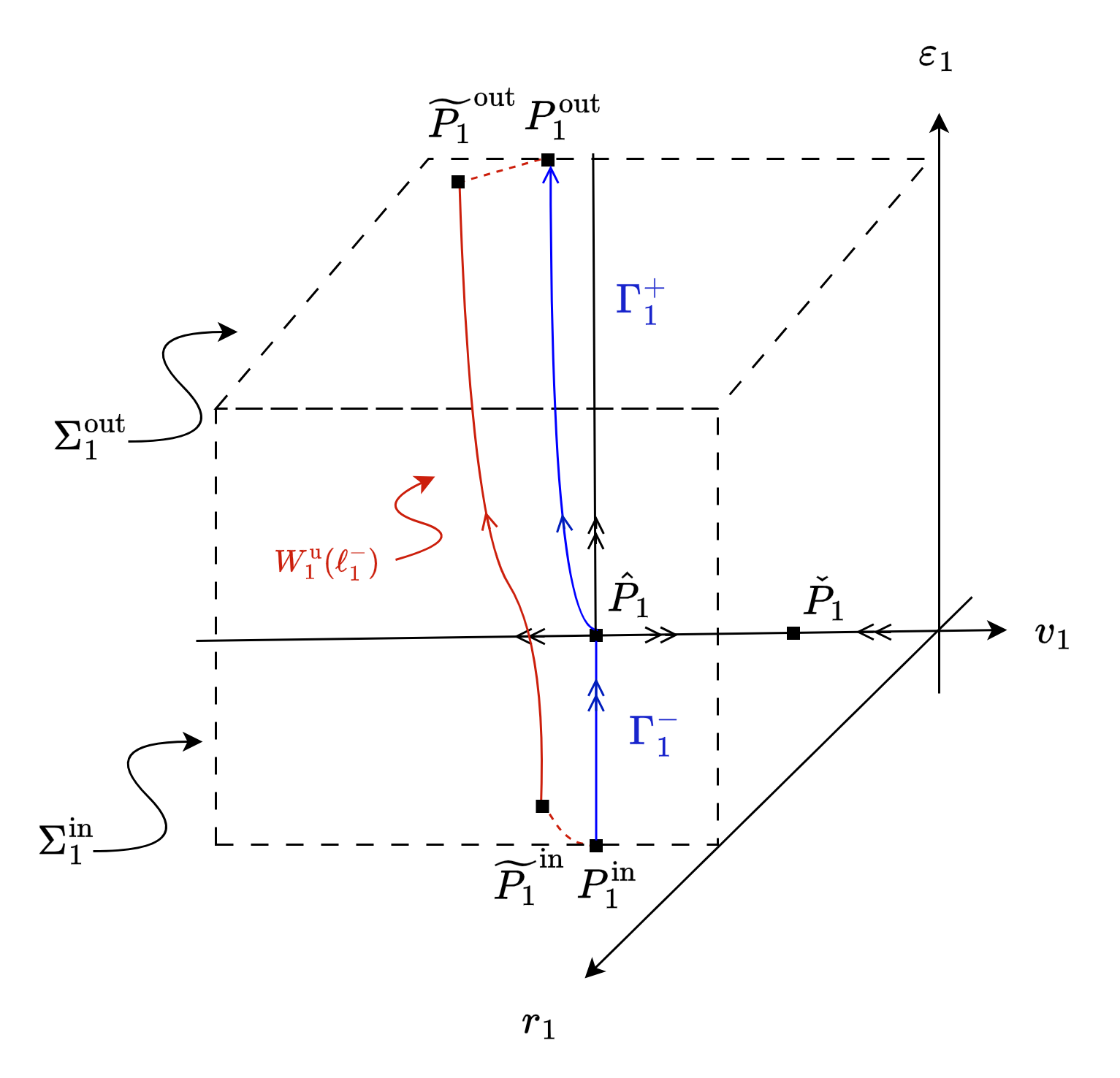}
    \caption{Geometry and dynamics in chart $K_1$ for $k>2$.}
    \label{fig:K>2}
\end{figure}

Next, we consider the portion $\Gamma_1^-$ of $\Gamma_1$. The limit as $\varepsilon_1\to 0^+$ in \eqref{K_1} gives the system of equations
\begin{equation}
    \begin{aligned}
        r_1'&=r_1v_1,\\
        v_1'&=-\Big(\frac{k}{2}+\frac{2}{k}\Big)v_1+kr_1v_1-(1-r_1)-v_1^2,
        \label{e_1-lim}
    \end{aligned}
\end{equation}
which is equivalent to \eqref{main} after blow-down.
\begin{lem}
    For $r_1\in(0,1]$ the system of equations in \eqref{e_1-lim} admits an explicit orbit that is given by 
    \begin{equation}
        \Gamma_{1}^-: \:\: v_1(r_1)=-\frac{k}{2}(1-r_1),
        \label{gamma_1^-push}
\end{equation}
with $v_1(1)=0$.
\end{lem}
\begin{proof}
    Recall that when $k\geq2$ and $c(0)=\frac{2}{k}+\frac{k}{2}=c_{\rm crit}$, we have an explicit orbit for \eqref{main} that is given by $V(U)=-\frac{k}{2}U(1-U)$, see Theorem~\ref{basic thm}. Transformation to chart $K_1,$ with $v=r_1v_1$ and $u=r_1$, yields the result.
\end{proof}

We conclude that $\Gamma_{1}^-$ is forward asymptotic to $\hat{P}_1$ and backward asymptotic to $Q_1^-$.
Moreover, $P_1^{\rm in}=\Gamma_{1}^-\cap \Sigma_{1}^{\rm in}=\big(r_0, \frac{k}{2}(r_0-1), 0\big)$, which completes the construction of the singular orbit $\Gamma_1$ in the case where $k>2$. \textcolor{black}{The geometry in chart $K_1$ is illustrated in Figure~\ref{fig:K>2} in that case.}

\subsection{Singular orbit \texorpdfstring{$\overline\Gamma$}{}}

We now combine the results of the previous two subsections to define the singular orbit $\overline\Gamma$ in $(\bar{u},\bar{v},\bar{\varepsilon})$-space.

\begin{prop}
    For any $k>0$, there exists a singular heteroclinic orbit $\overline{\Gamma}$ for Equations~\eqref{K_2before} and \eqref{K_1} that connects $Q_1^-$ to $Q_2^+$. 
    \label{propconst}
\end{prop}
\begin{proof}
    We first consider the case where $k\leq 2$, i.e., the pulled front propagation regime, which is analogous to the standard FKPP equation with a cut-off \cite{Dumortier_2007}. The orbit $\Gamma_{2}$ given by $v_2(u_2)=-2u_2$, see \eqref{Gamma2},  connects to $Q_2^+$ and intersects $\Sigma_{2}^{\rm in}$ in $P_{2}^{\rm in}=(1,-2,0).$
    Next, we apply the change of coordinates in \eqref{coordnatechange} to find $\kappa_{21}\big(P_{2}^{\rm in}\big)=P_{1}^{\rm out}=(0,-2,1)$. By construction, $\Gamma_{1}^+$ passes through $P_{1}^{\rm out}$ and is backward asymptotic to $P_1$, see \eqref{gamma_1+pull}.
    Similarly, $\Gamma_{1}^-$ is forward asymptotic to $P_1$ and backward asymptotic to $Q_1^-$, by  Proposition~\ref{trappingregionprop}. Therefore, we can now write $\overline{\Gamma}$ as the union of $\Gamma_{1}^-$, $\Gamma_{1}^+$, and $\Gamma_{2}$ with $Q_1^-$, $P_1$, and $Q_2^+$ in blown-up phase space, which proves the result for $k\leq 2$.

    Next, we consider the case where $k>2$, corresponding to the pushed propagation regime. Here, $\Gamma_{2}$ is given by $v_2(u_2)=-\big(\frac{k}{2}+\frac{2}{k}\big)u_2$, see \eqref{Gamma2}, which again connects to $Q_2^+$ and intersects $\Sigma_{2}^{\rm in}$ in $P_{2}^{\rm in}=\big(1,-\big(\frac{k}{2}+\frac{2}{k}\big),0\big)$. Applying the change of coordinates in \eqref{coordnatechange}, we find $\kappa_{21}\big(P_{2}^{\rm in}\big)=P_{1}^{\rm out}=\big(0,-\big(\frac{k}{2}+\frac{2}{k}\big),1\big)$.
    We know that $\Gamma_{1}^+$, constructed in \eqref{gamma_1+push},  passes through $P_{1}^{\rm out}$ and is backward asymptotic to $\hat{P}_1=\big(0, -\frac{k}{2}, 0\big)$.
    Similarly, $\Gamma_{1}^-$, defined in \eqref{gamma_1^-push}, is forward asymptotic to $\hat{P}_1$ and backward asymptotic to $Q_1^-$. Therefore, we can write $\overline{\Gamma}$ as the union of the orbits $\Gamma_{1}^-$, $\Gamma_{1}^+$, and $\Gamma_{2}$ with $Q_1^-$, $\hat{P_1}$, and $Q_2^+$ in blown-up space, which shows the result for $k>2.$
\end{proof}

\section{Proof of Theorem~\ref{Myfirstthm}}
\label{persistencesection}
In this section, we prove our main result, Theorem~\ref{Myfirstthm}. We first show that the singular orbit $\Gamma$, which is obtained from the orbit $\overline{\Gamma}$ constructed in Section~\ref{section 2} after blow-down, persists for $\varepsilon$ sufficiently small in Equation~\eqref{main}. Then, we derive the leading-order asymptotics of the correction $\Delta c(\varepsilon)$ to the critical speed $c(0)=c_{\rm crit}$ that is due to the cut-off. Finally, we illustrate our results numerically.

\subsection{Persistence of \texorpdfstring{$\Gamma$}{}}

The following result implies the existence of a unique front propagation speed $c(\varepsilon)$ for which there exists a critical heteroclinic orbit in~\eqref{main}. While the proof is similar to that of \cite[Proposition 3.1]{Dumortier_2007}, we give it here for completeness.

\begin{prop}
    For $\varepsilon\in(0, \varepsilon_0)$, with $\varepsilon_0$ sufficiently small, $k>0$, and $c$ close to $c(0)$, there exists a critical heteroclinic connection between $Q^-$ and $Q^+$ in Equation~\eqref{main} for a unique speed $c(\varepsilon)$ which depends on $k$. Furthermore, there holds $c(\varepsilon)\leq c(0)$.
    \label{persistenceprop}
\end{prop}
\begin{proof}
    We first analyse \eqref{K_2before} in the inner region, where $U< \varepsilon$. In particular, we are interested in the stable manifold $W_2^{\rm s}(\ell_2^+)$, which is given explicitly by $v_2(u_2)=-c(r_2)u_2$ when $r_2(=\varepsilon)>0$, for general values of $c$. For $r_2$ fixed, $W_2^{\rm s}(\ell_2^+)$ intersects $\Sigma_2^{\rm in}$ in the point $(1, v_2^{\rm in}, r_2)$, where $v_2^{\rm in}=-c(\varepsilon)$. From the definition of the blow-up transformation in \eqref{burger-blowup1}, we have that $V^{\rm in}=v_2^{\rm in}\varepsilon =-c(\varepsilon)\varepsilon\lesssim 0$, which implies $\frac{\partial V^{\rm in}}{\partial c}=-\varepsilon$. 

    We now consider the outer region, where $U>\varepsilon$.
    For general $c$, the dynamics in that region are governed by 
    \begin{equation}
        \begin{aligned}
            U'&=V,\\
            V'&=-cV+kUV-U(1-U),
        \end{aligned}
    \end{equation}
    recall \eqref{firstorder}.
    The intersection of the unstable manifold $W^{\rm u}(Q^-)$ of $Q^-$ with $\{U=\varepsilon\}$ can be written as the graph of an analytic function $V^{\rm out}(c, \varepsilon)$, with $\frac{\partial V^{\rm out}}{\partial c}>0$. A standard phase plane argument shows that $V^{\rm out}(c, \varepsilon)$ must be $\mathcal{O}(1)$ and negative for $c\lesssim c(0)$, which implies $V^{\rm in}>V^{\rm out}.$
    
    Finally, we consider the case where $c=c(0)$ and $\varepsilon>0$.
    First, we take $k<2$, in which case $c(0)=2$. The trapping region argument in Proposition~\ref{trappingregionprop} then shows that $V$ is bounded between the curves $\{V=0\}$ and $\{V=-U(1-U)\}$ and, hence, that $V^{\rm out}\geq -\varepsilon(1-\varepsilon)$ in $\{U=\varepsilon\}$. Therefore, we can conclude that \textcolor{black}{$V^{\rm in}=-2\varepsilon<-\varepsilon(1-\varepsilon)\leq V^{\rm out}$ for any $\varepsilon>0$}. Next, we take $k\geq2$, in which case the singular heteroclinic orbit is known explicitly as $V(U)=-\frac{k}{2}U(1-U)$. Therefore, we can write
    $V^{\rm out}(c(0), \varepsilon)=-\frac{k}{2}\varepsilon(1-\varepsilon)$, which again implies \textcolor{black}{$V^{\rm in}=-\big(\frac{k}{2}+\frac{2}{k}\big)\varepsilon<-\frac{k}{2}\varepsilon(1-\varepsilon)=V^{\rm out}$ for any $\varepsilon>0$}. 

    We conclude by observing that $W^{\rm s}(Q^+)$ and $W^{\rm u}(Q^-)$ must intersect in $\{U=\varepsilon\}$ for a unique value of $c(\varepsilon)\lesssim c(0)$, which follows from the implicit function theorem and the fact that $\frac{\partial V^{\rm out}}{\partial c}-\frac{\partial V^{\rm in}}{\partial c}>0$.    
\end{proof}

\textcolor{black}{
It follows from Proposition~\ref{persistenceprop} that a Heaviside cut-off reduces the critical front propagation speed in Equation~\eqref{Burger-FKPP1}; correspondingly, $\Delta c(\varepsilon)=c(0)-c(\varepsilon)$ must be positive for $\varepsilon$ sufficiently small.
}

\subsection{Leading-order asymptotics of \texorpdfstring{$\Delta c$}{}}

In this subsection, we derive the asymptotics of the correction $\Delta c$ to $c(0)$ to leading order in $\varepsilon$. Again, we distinguish between the pulled and pushed front propagation regimes in \eqref{main}.
\textcolor{black}{We first show the following preparatory result.
\begin{lem}\label{variational_lemma}
For $U$ and $V$ defined as in \eqref{firstorder}, $U\in[0, U_0]$ with $U_0>0$ sufficiently small, and any $k\geq 2$, there holds
    \begin{equation}
            \frac{\partial V}{\partial c}(U,c(0)) =\begin{cases}
          \frac{U(1-U+\ln{U})}{U-1} &\text{if}\:\: k= 2, \\ \frac{k^2}{k^2+4}U^{\frac 4{k^2}}(1-U){_2}F_1\big(1+\frac 4{k^2}, \frac 4{k^2}, 2+\frac 4{k^2}, 1-U\big)
       &\text{if}\:\: k>2,
    \end{cases}    
    \end{equation}
    where ${_2}F_1$ is the hypergeometric function, see, e.g.,~\cite[Section 15]{abramowitz+stegun}.
\end{lem}
\begin{proof}
    We rewrite \eqref{firstorder} with $U$ as the independent variable,
    \begin{equation}
        V\frac{dV}{dU}=-cV+kUV-U(1-U).
    \end{equation}
    Differentiation with respect to $c$ gives 
    \begin{equation}
        \frac{\partial V}{\partial c}\frac{\partial V}{\partial U}+V\frac{\partial}{\partial c}\frac{\partial V}{\partial U}=-V-c\frac{\partial V}{\partial c}+kU\frac{\partial V}{\partial c}.
    \end{equation}
    Evaluating at $V(U,c(0))=-\frac{k}{2}U(1-U)$ and making use of $\frac{\partial V}{\partial U}(U,c(0))=-\frac{k}{2}(1-2U)$, we find 
    \begin{equation}
        \frac{d}{d U}\left(\frac{\partial V}{\partial c}(U,c(0))\right)=-1+\frac{4}{k^2}\frac{1}{U(1-U)}\frac{\partial V}{\partial c}(U,c(0)).
        \label{varationalode}
    \end{equation}
    We note that \eqref{varationalode} is an ordinary differential equation for $\frac{\partial V}{\partial c}(U,c(0))$ in the variable $U$.
    For $k=2$, the unique solution that remains bounded as $U\to 1^-$ is given by 
    \begin{equation*}
        \frac{\partial V}{\partial c}(U,c(0))=\frac{U(1-U+\ln{U})}{U-1}.
    \end{equation*}
    For $k>2$, we can solve \eqref{varationalode} by variation of constants, which gives 
    \begin{equation*}
            \frac{\partial V}{\partial c}(U,c(0))=\beta (1-U)^{-\frac{4}{k^2}}U^{\frac{4}{k^2}}+(1-U)^{-\frac{4}{k^2}}U^{\frac{4}{k^2}}\Big[-\int_{1}^U(1-s)^{\frac{4}{k^2}}s^{-\frac{4}{k^2}}ds\Big],
    \end{equation*}
    for some constant of integration $\beta$ that is to be determined. We require that $\frac{\partial V}{\partial c}(U,c(0))\to 0$ when $U\to 1^-$. Therefore, $\beta=0$, since the second term goes to zero by L'H\^opital's Rule.
    Next, we make the substitution $s=1-\sigma$, which gives
    \begin{equation}
            \frac{\partial V}{\partial c}(U,c(0))=(1-U)^{-\frac{4}{k^2}}U^{\frac{4}{k^2}}\int_{0}^{1-U} \sigma^{\frac{4}{k^2}}(1-\sigma)^{-\frac{4}{k^2}}d\sigma.
            \label{intbetafunc}
    \end{equation}
The integral in \eqref{intbetafunc} is of the form of an Incomplete Beta function, see \cite[Section 6.6]{abramowitz+stegun}, which is defined by the expression 
\begin{equation*}
        B_x(a,b):=\int_0^{x}\sigma^{a-1}(1-\sigma)^{b-1}d\sigma.
\end{equation*}
Setting $x=1-U$, $a=1+\frac{4}{k^2}$, and $b=1-\frac{4}{k^2}$, we can write $\frac{\partial V}{\partial c}(U,c(0))$ in terms of $B_{1-U}\big(1+\frac{4}{k^2}, 1-\frac{4}{k^2}\big).$
Finally, the relation \cite[Equation~6.6.8 and Section~15]{abramowitz+stegun} 
\[
B_{x}(a,b)=a^{-1}x^a{_2}F_1(a, 1-b, a+1, x)
\]
implies
\begin{equation}
    \begin{aligned}
        \frac{\partial V}{\partial c}(U,c(0))&=\frac{k^2}{k^2+4}U^{4/k^2}(1-U){_2}F_1\big(1+4/k^2, 4/k^2, 2+4/k^2, 1-U\big),
    \end{aligned}
\end{equation}
which completes the proof.    
\end{proof}
}

\subsubsection{Pulled front propagation: \texorpdfstring{$k\leq 2$}{}}
\label{dc pulled}
We  first consider the case where $k\leq2$. Recall that the dynamics in chart $K_1$ are governed by the system of equations in~\eqref{K_1}. Our aim is to approximate the transition map $\Pi_1:\Sigma_{1}^{\rm in}\to\Sigma_{1}^{\rm out}$ under the flow of \eqref{K_1} for $\varepsilon\in(0,\varepsilon_0)$, with $\varepsilon_0>0$ sufficiently small. 

To that end, we first shift the equilibrium at $P_1=(0, -1, 0)$ to the origin via the transformation $V_1=v_1+1$, and we set $c=c(0)-\Delta c=2-\eta^2$. With these transformations, we can write \eqref{K_1} as
\begin{equation}
    \begin{aligned}
        r_1'&=-r_1(1-V_1),\\
        V_1'&=(2-\eta^2)(1-V_1)-kr_1(1-V_1)-1+r_1-(1-V_1)^2,\\
        \varepsilon_1'&=\varepsilon_1(1-V_1).
        \label{shifted}
    \end{aligned}
\end{equation}
Rescaling ``time" by dividing out a positive factor of $1-V_1$ from the right-hand sides in \eqref{shifted}, noting that the $\varepsilon_1$-equation decouples, and appending the trivial equation for $\eta$, we obtain
\begin{equation}
    \begin{aligned}
        \dot{r_1}&=-r_1,\\
        \dot{V}_1&=-\eta^2+\frac{(1-k)r_1+kr_1 V_1-V_1^2}{1-V_1}, \\
        \dot{\eta}&=0,
        \label{rescale}
    \end{aligned}
\end{equation}
where the overdot denotes differentiation with respect to the new independent variable \textcolor{black}{$\zeta$}.
\textcolor{black}{
\begin{remark}
The rescaling of ``time" in \eqref{shifted} is implicitly defined via $(1-V_1(\xi))\frac{d}{d\zeta}=\frac{d}{d\xi}$, and merely affects the parametrisation of solutions while leaving the phase portrait unchanged.
\end{remark}
}

We have the following result, which can be shown in close analogy to \cite[Proposition 3.2]{Dumortier_2007}.
\begin{lem}
There exists a normal form transformation $(r_1, V_1, \eta)\to(S(r_1, V_1, \eta), W(r_1, V_1, \eta), \eta)$ that transforms E\-qua\-tion \eqref{rescale} to 
\begin{equation}
    \begin{aligned}
        \dot{S}&=-S,\\
        \dot{W}&=-\eta^2-\frac{W^2}{1-W}, \\
        \dot{\eta}&=0.
        \label{basicnormal}
    \end{aligned}
\end{equation}
That transformation respects the invariance of $\{r_1=0\}$ and $\{\eta=\eta_0\}$, for any $\eta_0\in\mathbb{R}$.
\label{normal_form}
\end{lem}
\begin{proof}
    The statement follows from \cite[Theorem~1]{0df3afe5-6a08-31cf-8efe-d9042a9b599b}.
\end{proof}

We note that the only resonant terms in \eqref{rescale} are of the form $V_1^n$, for $n\geq2$. Therefore, all other terms can be removed via a sequence of smooth near-identity transformations \textcolor{black}{of the form $W=V_1+(1-k)r_1+\frac{(3k-4)(k-1)}{2}r_1^2+(3-2k)r_1V+(2k-3)r_1\eta^2+\mathcal{O}(3)$, where $\mathcal{O}(3)$ denotes terms of at least order $3$ in $r_1$, $V_1$, and $\eta$.}

The normal form in \eqref{basicnormal} is identical to the one stated in \cite[Equation~(34)]{Dumortier_2007}. Moreover, the analysis in \cite{Dumortier_2007} shows that the correction $\Delta c$ to $c(0)$ is given by $\eta^2=\frac{\pi^2}{(\ln\varepsilon)^2}+\mathcal{O}[(\ln\varepsilon)^{-3}]$ to leading order, as well as that it is independent of the transformed coordinates $W^{\rm in}$ and $W^{\rm out}$ of the entry and exit points $P_1^{\rm in}$ and $P_1^{\rm out}$, respectively, \textcolor{black}{provided $W^{\rm in}>0$ and $W^{\rm out}<0$ as $\eta\to0^+$,} following the normal form transformation in Lemma~\ref{normal_form}. \textcolor{black}{The next result is required to approximate $\eta(\varepsilon)$ in the two cases where $k<2$ and $k=2$, respectively.}

\textcolor{black}{\begin{lem}
    For $\eta, \varepsilon$, and $r_0$ sufficiently small, the points $\widetilde{P}_1^{\rm in}=\big(r_0, W^{\rm in}, \frac{\varepsilon}{r_0}\big)$ and $\widetilde{P}_1^{\rm out}=(\varepsilon, W^{\rm out}, 1)$ satisfy
    \begin{equation*}
        W^{\rm in}\begin{cases}
        >0 & \text{for}\:\: k<2, \\
        \to 0^- & \text{for}\:\: k=2
        \end{cases}
    \end{equation*}
    and
    \begin{equation*}
        W^{\rm out}<0\quad\text{as}\;\eta\to0^+
    \end{equation*}
    for $k\leq 2$.
    \label{lem: pulled W in and out}
\end{lem}
\begin{proof}
    For $k<2$, recall from Proposition~\ref{trappingregionprop} that $-U(1-U)<V(U)<0$ for $0<U<1$. Next, recall that, by \eqref{burger-blowup1}, $U=r_1$ and $V=r_1v_1$, which, in combination, yields $v_1>-1+r_1$. Then, we make use of $V_1=v_1+1$ and the fact that $r_1=r_0$ in $\Sigma_1^{\rm in}$ to find $V_1>r_0$. Finally, it follows from the normal form transformation defined in Lemma~\ref{normal_form} that $W^{\rm in}>(2-k)r_0+\mathcal{O}(r_0^2)>0$. 
    We now consider the case where $k=2$. As $W^{\rm u}(Q^-)$ is analytic in $U$ and $c$, we can write
    \begin{equation*}
        \begin{aligned}
            V(U,c)&=\sum_{j=0}^{\infty}\frac{1}{j!}\frac{\partial^j V}{\partial c^j}(U,c(0))(-\eta^2)^j\\
            &=-U(1-U)-\frac{U(1-U+\ln{U})}{U-1}\eta^2+\mathcal{O}(\eta^4),
        \end{aligned}
    \end{equation*}
    by Proposition~\ref{trappingregionprop} and Lemma~\ref{variational_lemma}. Making use of $U=r_1$, $V=r_1v_1$, and $V_1=v_1+1$, as well as of $r_1=r_0$ in $\Sigma_1^{\rm in}$, we find 
    \begin{equation*}
        V_1^{\rm in}=r_0-\frac{1-r_0+\ln{r_0}}{r_0-1}\eta^2+\mathcal{O}(\eta^4).
    \end{equation*}
    Next, it follows from the normal form transformation defined in Lemma~\ref{normal_form} that $W^{\rm in}=V_1^{\rm in}-r_0+r_0^2-r_0V_1^{\rm in}+r_0\eta^2+\mathcal{O}(3)$, which implies $W^{\rm in}=(1+\ln{r_0})\eta^2+\mathcal{O}(r_0^2\eta^2, \eta^4)\to 0^-$ as $\eta\to0^+$.
    \\
    Finally, a simple calculation shows that $v_1^{\rm out}=-c(\varepsilon)=-2+\eta^2$; therefore, as $V_1=v_1+1$ and $W=V_1+(1-k)r_1+\mathcal{O}(2)$ and as $r_1=\varepsilon$ in $\Sigma_1^{\rm out}$, we find $W^{\rm out}=-1+\eta^2+\mathcal{O}(\varepsilon)<0$ for $\varepsilon>0$ sufficiently small.
\end{proof}
In summary, it follows from the analysis in \cite{Dumortier_2007} that $\eta=-\frac{\pi}{\ln{\varepsilon}}+o[(\ln{\varepsilon})^{-1}]$ for $k<2$. Similarly, the analysis in \cite{Popović_2011, article2} implies that $\eta=-\frac{\pi}{2\ln{\varepsilon}}+o[(\ln{\varepsilon})^{-1}]$ when $k=2$, which completes the proof of Theorem~\ref{Myfirstthm} in the pulled propagation regime.}

\begin{remark}
Setting $k=0$ in Theorem~\ref{Myfirstthm}, we recover the main result from \cite[Theorem~1.1]{Dumortier_2007}, as is to be expected.
\end{remark}

\subsubsection{Pushed front propagation: \texorpdfstring{$k>2$}{}}

The pushed propagation regime where $k>2$ is significantly more involved algebraically than the pulled regime discussed in the previous subsection.

Our aim is again to approximate the transition map $\Pi_1:\Sigma_{1}^{\rm in}\to\Sigma_{1}^{\rm out}$ under the flow of \eqref{K_1}. Now, the point $\hat{P}_1=\big(0, -\frac{k}{2}, 0\big)$ is shifted to the origin via the transformation $V_1=v_1+\frac{k}{2}$; moreover, we write $c=c(0)-\Delta c=\frac{k}{2}+\frac{2}{k}-\Delta c$.
The resulting system of equations is given by 
\begin{equation}
    \begin{aligned}
        r_1'&=-r_1\Big(\frac{k}{2}-V_1\Big),\\
        V_1'&=-\Delta c\Big(\frac{k}{2}-V_1\Big)+r_1\Big(1-\frac{k^2}{2}+kV_1\Big)+\Big(\frac{k}{2}-\frac{2}{k}\Big)V_1-V_1^2,\\
        \varepsilon_1'&=\varepsilon_1\Big(\frac{k}{2}-V_1\Big).
    \end{aligned}
\end{equation}
Next, we rescale ``time" by a (positive) factor of $\frac{k}{2}-V_1$, \textcolor{black}{with $\big(\frac{k}{2}-V_1(\xi)\big)\frac{d}{d\zeta}=\frac{d}{d\xi}$,} which yields
\begin{equation}
    \begin{aligned}
        \dot{r_1}&=-r_1,\\
        \dot{V_1}&=-\Delta c+\frac{r_1\big(1-\frac{k^2}{2}+kV_1\big)+\big(\frac{k}{2}-\frac{2}{k}\big)V_1-V_1^2}{\frac{k}{2}-V_1},\\
        \dot{\varepsilon_1}&=\varepsilon_1.
        \label{K_1_timerescale}
    \end{aligned}
\end{equation}
We note that the equation for $\varepsilon_1$ in \eqref{K_1_timerescale} again decouples.
Finally, we separate the $r_1$-dependent terms in the $V_1$-equation in \eqref{K_1_timerescale}, and we append the trivial equation for $\Delta c$:
\begin{equation}
    \begin{aligned}
        \dot{r_1}&=-r_1,\\
        \dot{V_1}&=-\Delta c+r_1 \frac{1-\frac{k^2}{2}+kV_1}{\frac{k}{2}-V_1}+\frac{\left(\frac{k}{2}-\frac{2}{k}\right)V_1-V_1^2}{\frac{k}{2}-V_1},\\
        \dot{\Delta c}&=0.
        \label{expanded}
    \end{aligned}
\end{equation}
For the linearisation of \eqref{expanded} at the origin, we obtain the eigenvalues $-1$, $1-\frac{4}{k^2}$, and $0$. It is straightforward to show that the monomial $r_1V_1^j$ in \eqref{expanded} can only be resonant for integer-valued $j=\frac{2-4/k^2}{1-4/k^2}$. In particular, the lowest-order resonance is realised at order $4$, since $1(-1)+3\big(1-\frac{4}{k^2}\big)=1-\frac{4}{k^2}$ when $k=2\sqrt{2}$, corresponding to the fourth-order monomial $r_1 V_1^3$. To approximate $\Pi_1$, we hence first derive a normal form for \eqref{expanded} by eliminating all non-resonant $r_1$-dependent terms via a sequence of near-identity transformations.

\begin{lem}
    There exists a sequence of smooth transformations that transforms Equation~\eqref{expanded} to 
    \begin{equation}
        \begin{aligned}
            \dot{r_1}&=-r_1,\\
            \dot{W}&=-\Delta c+\frac{\big(\frac{k}{2}-\frac{2}{k}\big)W-W^2}{\frac{k}{2}-W}+\mathcal{O}(r_1 W^j),\\
            \dot{\Delta c}&=0,
            \label{k>2 norm}
        \end{aligned}
    \end{equation}
    with $j\geq 3$. Specifically, that sequence is composed of the transformation $V_1=\frac{k}{2}r_1+Z$ in \eqref{expanded}, followed by the near-identity transformation $Z=W+\frac{4}{k^2}r_1W$ and, finally, a sequence of smooth near-identity transformations. 
    \label{normal_form_1}
\end{lem}

\begin{proof}
    The existence of such a transformation follows from standard normal form theory \cite{hirsch2006invariant}, as the lowest-order potentially resonant monomial in \eqref{expanded} is of the form $r_1V_1^3$ for $k=2\sqrt{2}$. All higher-order non-resonant terms can be removed by a sequence of smooth near-identity transformations.
\end{proof}

Next, \textcolor{black}{we approximate $\widetilde{P}_{1}^{\rm in}$ and $\widetilde{P}_{1}^{\rm out}$}, which are the entry and exit points in $\Sigma_1^{\rm in}$ and $\Sigma_1^{\rm out}$, respectively, under $\Pi_1$, to a sufficiently high order in $\Delta c$, $\varepsilon$, and $r_0$. 

\begin{lem}
    For $k>2$ and $\Delta c$ and $\varepsilon$ sufficiently small, the points $\widetilde{P}_1^{\rm in}=\big(r_0, W^{\rm in}, \frac{\varepsilon}{r_0}\big)$ and $\widetilde{P}_1^{\rm out}=(\varepsilon, W^{\rm out}, 1)$ satisfy 
    \begin{equation}
        \begin{aligned}
            W^{\rm in}=\nu(r_0)\Delta c+\mathcal{O}(\Delta c^2, r_0^{4/k^2}\Delta c)\quad\text{and}\quad  
            \textcolor{black}{W^{\rm out}=-\frac{2}{k}+\Delta c+\mathcal{O}(\varepsilon),}
        \end{aligned}
    \end{equation}
    where 
    \begin{equation}\label{def:vr0}
        \nu(r_0)=-\frac{k^2}{k^2+4}r_0^{4/k^2-1}(1-r_0){_2}F_1\big(1+4/k^2, 4/k^2, 2+4/k^2, 1-r_0\big).
    \end{equation}
    \label{in/outlemma}
\end{lem}

\begin{proof}
    Recall that, by \eqref{burger-blowup1}, $r_1v_1=r_2v_2$, which implies $v_1^{\rm out}=v_2^{\rm in}=-c(\varepsilon)=-\big(\frac{k}{2}+\frac{2}{k}\big)+\Delta c$.
    As the point $P_1$ was shifted to the origin via the transformation $V_1=v_1+\frac{k}{2},$ we have $V_1^{\rm out}=-\frac{2}{k}+\Delta c$.
    The normal form transformation given by Lemma~\ref{normal_form_1} implies that $W=V_1-\frac{k}{2}r_1+\mathcal{O}(r_1 V_1).$ Moreover, $r_1=\varepsilon$ in $\Sigma_1^{\rm out}$;
    therefore, \textcolor{black}{$W^{\rm out}=-\frac{2}{k}+\Delta c+\mathcal{O}(\varepsilon)$.}

    We now consider $W^{\rm in}$. As $W^{\rm u}(Q^-)$ is analytic in $U$ and $c$, we can write 
    \begin{equation*}
        \begin{aligned}
            V(U,c)&=\sum_{j=0}^{\infty}\frac{1}{j!}\frac{\partial^j V}{\partial c^j}(U,c(0))(-\Delta c)^j\\
            &=-\frac{k}{2}U(1-U)-\frac{k^2}{k^2+4}U^{4/k^2}(1-U) {_2}F_1\big(1+4/k^2, 4/k^2, 2+4/k^2, 1-U\big)\Delta c+\mathcal{O}(\Delta c^2),
        \end{aligned}
    \end{equation*}
    by Lemma~\ref{variational_lemma}.
    Next, we make use of $U=r_1$, $V=r_1v_1$, and the fact that $r_1=r_0$ in $\Sigma_1^{\rm in}$, as well as of the transformation $V_1=v_1+\frac{k}{2}$, to obtain
    \begin{equation*}
    \begin{aligned}
            V_1^{\rm in}&=\frac{k}{2}r_0-\frac{k^2}{k^2+4}r_0^{4/k^2-1}(1-r_0){_2}F_1\big(1+4/k^2, 4/k^2, 2+4/k^2, 1-r_0\big)\Delta c+\mathcal{O}(\Delta c^2).
        \end{aligned}
    \end{equation*}
    Finally, since $W=V_1-\frac{k}{2}r_1+\mathcal{O}(r_1V_1)$, we have 
    \begin{equation}
        \begin{aligned}
            W^{\rm in}&=-\frac{k^2}{k^2+4}r_0^{4/k^2-1}(1-r_0) {_2}F_1\big(1+4/k^2, 4/k^2, 2+4/k^2, 1-r_0\big)\Delta c+\mathcal{O}(\Delta c^2, r_0^{4/k^2}\Delta c)\\
            &=\nu(r_0)\Delta c+\mathcal{O}(\Delta c^2, r_0^{4/k^2}\Delta c),
            \label{W^in final}
        \end{aligned}
    \end{equation}    
    where $\nu(r_0)$ is as defined in~\eqref{def:vr0}. 
    \textcolor{black}{In particular, the invariance of $\{W=0\}$ for $\Delta c=0$ in the normal form, \\
    Equation~\eqref{k>2 norm}, implies that the error term in \eqref{W^in final} has to be proportional to $\Delta c$, as stated.}
\end{proof}
\textcolor{black}{
\begin{remark}\label{negative}
Lemma~\ref{in/outlemma} implies that $W^{\rm in}$ and $W^{\rm out}$ are both negative for $\Delta c$ and $r_0$ sufficiently small. In particular, $\nu(r_0)$, as defined in \eqref{def:vr0}, is negative, which follows from the identity ${_2}F_1(a,b,c,1)=\frac{\Gamma(c)\Gamma(c-a-b)}{\Gamma(c-a)\Gamma(c-b)}$; see, e.g.,~\cite[Equation 15.1.20]{abramowitz+stegun}. Incidentally, that identity is valid for $\Re(c-a-b)>0$ which, for $a=1+4/k^2$, $b=4/k^2$, and $c=2+4/k^2$, is equivalent to requiring $k>2$.
\end{remark}
}

Instead of integrating the ``full" normal form in \eqref{k>2 norm} to determine the leading-order asymptotics of $\Delta c(\varepsilon)$, we will consider the simplified equations that are obtained by omitting the higher-order $\mathcal{O}(r_1W^j)$-terms with $j\geq 3$ therein: 

\textcolor{black}{
\begin{equation} 
    \dot{\widehat{W}}=-\Delta c+\frac{\left(\frac{k}{2}-\frac{2}{k}\right)\widehat{W}-\widehat{W}^2}{\frac{k}{2}-\widehat{W}}.
    \label{simplifed NF}
\end{equation}}

We now show that, to leading order, the asymptotics of $\Delta c(\varepsilon)$ is not affected by the omission of the $\mathcal{O}(r_1W^j)$-terms in~\eqref{k>2 norm}. In our proof, we make the {\it a~priori} assumption that $\Delta c=\mathcal{O}\big(\varepsilon^{1-4/k^2}\big)$, which we then show to be consistent in Proposition~\ref{Dc prop}.

\begin{lem} Let $\zeta \in [0, \zeta^{\rm out}]$, with $W^{\rm in}$ and $W^{\rm out}$ defined as in Lemma~\ref{in/outlemma}, let $k>2$, and let $\varepsilon\in(0,\varepsilon_0)$, with $\varepsilon_0>0$ sufficiently small. Then, for $W^{\rm in}=W(0)=\widehat{W}^{\rm in}$, we have
\[
 \big|W^{\rm out}-\widehat{W}^{\rm out}\big|=\mathcal{O}(\varepsilon^\kappa),
\] 
where \textcolor{black}{$\kappa>\frac{1}{4}$}.
\label{lem normal_form}
\end{lem}

\textcolor{black}{
\begin{proof}
Considering the difference between the equations for $W$ and $\widehat W$ in \eqref{k>2 norm} and \eqref{simplifed NF}, respectively, we find 
\[
 \frac{d}{d\zeta}\big[W(\zeta)-\widehat W(\zeta)\big]=\Bigg[ 1- \frac 4{(k-2W)(k-2\widehat W)}\Bigg]\big[W(\zeta)-\widehat W(\zeta)\big]+\mathcal{O}(r_1W^j),
\]
where $j\geq 3$. Since $W,\widehat{W}\in[W^{\rm out}, W^{\rm in}]$,
and since $W^{\rm out}=-\frac{2}{k}+\Delta c+\mathcal{O}(\varepsilon)$, by Lemma~\ref{in/outlemma}, with $\Delta c=\mathcal{O}\big(\varepsilon^{1-4/k^2}\big)$ positive, there exists $\varepsilon_0>0$ sufficiently small such that $-\frac{2}{k}\leq W^{\rm out}$ for $\varepsilon\in(0, \varepsilon_0)$. Therefore, we can estimate $(k-2W)(k-2\widehat{W})\leq \big(k+\frac{4}{k}\big)^2$.
Thus, for $W^{\rm in}=W(0)=\widehat{W}^{\rm in}$, an application of the Gr\o nwall inequality yields 
 \begin{align}
     \big|W(\zeta)-\widehat W(\zeta)\big|&\leq C e^{[1-4k^2/(k^2+4)^2]\zeta}\int_0^{\zeta} e^{-[1-4k^2/(k^2+4)^2]s}|r_1 W^j|\,ds,
     \label{Gron-application}
\end{align}
where $C$ is a generic constant.
Next, we estimate the $\mathcal{O}(r_1W^j)$-terms in \eqref{Gron-application}: linearising \eqref{k>2 norm} and solving by variation of constants, with $W(0)=W^{\rm in}=\omega(\Delta c,r_0)\,\Delta c$ for $\omega(\Delta c,r_0)=\nu(r_0)+\mathcal{O}(\Delta c,r_0^{4/k^2})$, recall Lemma~\ref{in/outlemma}, we find
\[
    W(\zeta)=W^{\rm in}e^{(1-4/k^2)\zeta}-\frac{\Delta c}{1-\frac{4}{k^2}}\big[e^{(1-4/k^2)\zeta}-1\big]+
    \int_0^\zeta e^{(1-4/k^2)(\zeta-s)}\mathcal{O}(W^2,r_1W^j)\,ds.
\]
Now, a bootstrapping argument shows that $|W(\zeta)|\leq C\Delta c e^{(1-4/k^2)\zeta}$ for $\zeta\in [0,\zeta^\ast]$, where $\zeta^\ast:=\frac{1}{1-4/k^2}\ln\frac{\mu}{\Delta c}$ is the largest ``time" such that $\Delta c e^{(1-4/k^2)\zeta}\leq\mu$, with $\mu$ a small, but fixed constant. To that end, we note that, trivially, $\big|W^{\rm in}e^{(1-4/k^2)\zeta}-\frac{\Delta c}{1-4/k^2}\big[e^{(1-4/k^2)\zeta}-1\big]\big|\leq C\Delta c e^{(1-4/k^2)\zeta}$; then, we recall that $W\in[W^{\rm out},W^{\rm in}]$, with $-\frac2k\leq W^{\rm out}$ and $W^{\rm in}\leq 0$ due to $\omega(\Delta c,r_0)$ being negative for $\Delta c$ and $r_0$ sufficiently small, by Remark~\ref{negative}, which implies
\[
\int_0^\zeta e^{(1-4/k^2)(\zeta-s)}\mathcal{O}(W^2)\,ds \leq C(\Delta c)^2 e^{(1-4/k^2)\zeta}\int_0^\zeta e^{(1-4/k^2)s}\,ds\leq C\Delta c e^{(1-4/k^2)\zeta}.
\]
Here, we have used that $\Delta c e^{(1-4/k^2)\zeta}\leq\mu$, by the definition of $\zeta^\ast$. Finally, in
\[
\int_0^\zeta e^{(1-4/k^2)(\zeta-s)}\mathcal{O}(r_1W^j)\,ds\leq C(\Delta c)^je^{(1-4/k^2)\zeta}\int_0^\zeta e^{[(1-4/k^2)(j-1)-1]s}\,ds,
\]
we recall that the $\mathcal{O}(r_1W^j)$-terms in \eqref{k>2 norm} can only be resonant for integer-valued $j=\frac{2-4/k^2}{1-4/k^2}$, which gives 
$e^{[(1-4/k^2)(j-1)-1]s}\equiv 1$ and $(\Delta c)^{j-1}=\mathcal{O}(\varepsilon)$, since $\Delta c=\mathcal{O}\big(\varepsilon^{1-4/k^2}\big)$, by assumption. Hence, it follows that $\int_0^\zeta e^{(1-4/k^2)(\zeta-s)}\mathcal{O}(r_1W^j)\,ds\leq C\varepsilon|\ln\varepsilon|\Delta ce^{(1-4/k^2)\zeta}$, completing the bootstrap.
In sum, the integral in \eqref{Gron-application} can therefore be bounded as follows: on $[0,\zeta^\ast]$, we may write
\begin{align*}
\int_0^{\zeta^\ast}e^{-[1-4k^2/(k^2+4)^2]s}|r_1 W^j|\,ds 
&\leq C(\Delta c)^j\int_0^{\zeta^\ast}e^{[4k^2/(k^2+4)^2+(1-4/k^2)j-2]s}\,ds \\
&=C(\Delta c)^{(2-4/k^2)/(1-4/k^2)}\int_0^{\zeta^\ast}e^{-32(k^2+2)/[k^2(k^2+4)^2]s}\,ds=\mathcal{O}\big(\varepsilon^{2-4/k^2}\big),
\end{align*}
where we have again used $\Delta c=\mathcal{O}\big(\varepsilon^{1-4/k^2}\big)$ and $j=\frac{2-4/k^2}{1-4/k^2}$. 
On $[\zeta^\ast,\zeta^{\rm out}]$, we note that $\Delta\zeta:=\zeta^{\rm out}-\zeta^\ast=\mathcal{O}(1)$, which, together with $r_1(s)\leq r_0e^{-(\zeta^{\rm out}-\Delta\zeta)}=e^{\Delta\zeta}\varepsilon$, $e^{-[1-4k^2/(k^2+4)^2]s}\leq e^{-[1-4k^2/(k^2+4)^2](\zeta^{\rm out}-\Delta\zeta)}$, and $|W|=\mathcal{O}(1)$, implies
\[
\int_{\zeta^\ast}^{\zeta^{\rm out}}e^{-[1-4k^2/(k^2+4)^2]s}|r_1 W^j|\,ds=\mathcal{O}\big(\varepsilon^{2-4k^2/(k^2+4)^2}\big).
\]
Since $\frac{4k^2}{(k^2+4)^2}<\frac{4}{k^2}$, it hence follows that $\mathcal{O}\big(\varepsilon^{2-4k^2/(k^2+4)^2}\big)=o(\varepsilon^{2-4/k^2})$, which implies that the contribution from $[\zeta^\ast,\zeta^{\rm out}]$ is negligible.  
In sum, we conclude that,
for $\zeta=\zeta^{\rm out}$, $j=\frac{2-4/k^2}{1-4/k^2}$, and $\varepsilon$ and $r_0$ sufficiently small, \eqref{Gron-application} yields
\[
\big|W^{\rm out}-\widehat{W}^{\rm out}\big|=\mathcal{O}\big(\varepsilon^{4k^2/(k^2+4)^2+1-4/k^2}\big)=\mathcal{O}\big(\varepsilon^{\kappa(k)}\big),
\]
as the exponential term in \eqref{Gron-application} satisfies $e^{[1-4k^2/(k^2+4)^2]\zeta^{\rm out}}=\mathcal{O}\big(\varepsilon^{4k^2/(k^2+4)^2-1}\big)$; here, $\kappa(k)=\frac{1}{4}+\frac{(k^2-4)(3k^4+36k^2+64)}{4k^2(k^2+4)^2}$. We note that $\kappa(k)>\frac{1}{4}$ 
is monotonically increasing for $k\in(2, \infty)$, with $\lim_{k\to \infty}\kappa(k)=1$.
Hence, it follows that $\big|W^{\rm out}-\widehat{W}^{\rm out}\big|=\mathcal{O}(\varepsilon^\kappa)\to 0$ as $\varepsilon\to0$ with $\kappa>\frac{1}{4}$, as stated.
\end{proof}}

We can now solve \eqref{simplifed NF} by separation of variables,
\begin{equation}
    \begin{aligned}
        \frac{-2kW+k^2}{-2kW^2+W(k^2+2\Delta c k-4)-\Delta ck^2}dW&=d\zeta,
        \label{integrated_normal_form}
    \end{aligned}
\end{equation}
where we have omitted overhats from $\widehat{W}$ for simplicity of notation.
%
Integration of \eqref{integrated_normal_form} gives
\begin{equation}
    \begin{aligned}
        & \zeta^{\rm out}-\zeta^{\rm in}-\frac{1}{2}\ln\big|-2kW^2+(k^2+
        2k\Delta c -4)W- k^2\Delta c\big|\Big\lvert_{W^{\rm in}}^{W^{\rm out}}\\
        &\; -\frac{\frac{k^2}{2}+2-k\Delta c}{\sqrt{(k^2+2k\Delta c -4)^2-8 k^3\Delta c}} \\
        &\; \times\ln\bigg|\frac{-4kW+k^2+2k\Delta c -4-\sqrt{(k^2+2k\Delta c -4)^2-8 k^3\Delta c}}{-4kW+k^2+2k\Delta c -4+\sqrt{(k^2+2k\Delta c -4)^2-8 k^3\Delta c}}\bigg|\bigg\lvert_{W^{\rm in}}^{W^{\rm out}}=0.
        \label{int_2}
    \end{aligned}
\end{equation}

Recall that, as $\varepsilon_1(\zeta)=\frac{\varepsilon}{r_0}e^{\zeta}$, we have $\zeta^{\rm in}=0$ and $\zeta^{\rm out}=-\ln\frac{\varepsilon}{r_0}$ in \eqref{int_2}. Moreover, by Lemma~\ref{in/outlemma}, 
\begin{align*}
W^{\rm in}=\nu(r_0)\Delta c+\mathcal{O}(\Delta c^2, r_0^{4/k^2}\Delta c)\quad\text{and}\quad
\textcolor{black}{W^{\rm out}=-\frac{2}{k}+\Delta c+\mathcal{O}(\varepsilon).}
\end{align*}

We now proceed as follows: given \eqref{int_2}, we derive a necessary condition on $\Delta c$ which will determine the leading-order asymptotics thereof in $\varepsilon$.

We begin by substituting our estimates for $W^{\rm in}$ and $W^{\rm out}$ into the first logarithmic term in \eqref{int_2}, which gives
\begin{equation}
    \begin{aligned}
        \frac{1}{2}&\ln\big|-2k{(W^{\rm in})}^2+(k^2+2k\Delta c -4)W^{\rm in}-k^2\Delta c\big|\\
        &=\frac{1}{2}\ln\big|(k^2-4)\nu(r_0)\Delta c-k^2\Delta c +\mathcal{O}(\Delta c^2, r_0^{4/k^2}\Delta c)\big|
    \end{aligned}
\end{equation}
and
\begin{equation}
    \begin{aligned}
        \frac{1}{2}&\ln\big|-2k{(W^{\rm out})}^2+(k^2+2k\Delta c -4)W^{\rm out}-k^2\Delta c\big|\\
        &=\frac{1}{2}\ln\big|-2k+\mathcal{O}(\Delta c, \varepsilon)\big|,
    \end{aligned}
\end{equation}
respectively.

Now, we expand the rational function multiplying the second logarithmic term in \eqref{int_2} as
\begin{equation}
    -\frac{\frac{k^2}{2}+2-k\Delta c}{\sqrt{(k^2+2k\Delta c -4)^2-8 k^3\Delta c }}=-\frac{k^2+4}{2(k^2-4)}-\frac{16 k^3}{(k^2-4)^3}\Delta c\textcolor{black}{+\mathcal{O}(\Delta c^2),}
\end{equation}
and we write the argument of the logarithm therein as
\begin{equation}
    \begin{aligned}
        &\frac{-4kW+k^2+2k\Delta c -4-\sqrt{(k^2+2k\Delta c -4)^2-8 k^3\Delta c}}{-4kW+k^2+2k\Delta c -4+\sqrt{(k^2+2k\Delta c -4)^2-8 k^3\Delta c}}\\
        &\quad=-1+2\frac{-4kW+k^2+2k\Delta c -4}{-4kW+k^2+2k\Delta c -4+\sqrt{(k^2+2k\Delta c -4)^2-8 k^3\Delta c}}.
        \label{rational break}
    \end{aligned}
\end{equation}

Substituting the estimate for $W^{\rm in}$ into \eqref{rational break}, we have
\begin{equation}
    \begin{aligned}
        -1 &+2\frac{(-4k\nu(r_0)+2k)\Delta c+k^2-4}{-4k\nu(r_0)\Delta c+k^2+2k\Delta c -4+\sqrt{(k^2+2k\Delta c -4)^2-8 k^3\Delta c}}+\mathcal{O}(\Delta c^2, r_0^{4/k^2}\Delta c)\\
        &=\bigg[\frac{2k^3}{(k^2-4)^2}-\frac{2k\nu(r_0)}{k^2-4}\bigg]\Delta c+\mathcal{O}(\Delta c^2, r_0^{4/k^2}\Delta c).
    \end{aligned}
\end{equation}

Similarly, we can use our estimate for $W^{\rm out}$ in \eqref{rational break} to obtain
\begin{equation}
    \begin{aligned}
        &-1+2\frac{4+k^2}{2k^2}+\mathcal{O}(\Delta c, \varepsilon)=\frac{4}{k^2}+\mathcal{O}(\Delta c, \varepsilon).
    \end{aligned}
\end{equation}

Summarising the above calculations, we can write \eqref{int_2}
as 
\begin{equation}
    \begin{aligned}
        &-\ln\frac{\varepsilon}{r_0}+\frac{1}{2}\ln\big|(k^2-4)\nu(r_0)\Delta c- k^2\Delta c+\mathcal{O}(\Delta c^2, r_0^{4/k^2}\Delta c)\big|-\frac{1}{2}\ln|-2k+\mathcal{O}(\Delta c, \varepsilon)|\\
        &\quad 
        - \bigg[\frac{k^2+4}{2(k^2-4)}+\mathcal{O}(\Delta c)\bigg]\bigg[-\ln \bigg|\bigg[\frac{2 k^3}{(k^2-4)^2}-\frac{2k \nu(r_0)}{k^2-4}\bigg]\Delta c +\mathcal{O}(\Delta c^2, r_0^{4/k^2}\Delta c)\bigg|\\
        &\quad+\ln\bigg|\frac{4}{k^2}+\mathcal{O}(\Delta c, \varepsilon)\bigg|\bigg]=0.
        \label{Normalformsub}
    \end{aligned}
\end{equation}
Now, we exponentiate \eqref{Normalformsub} to obtain 
\begin{equation}
    \begin{aligned}
        \bigg(\frac{\varepsilon}{r_0}\bigg)^2&=\frac{[k^2-(k^2-4)\nu(r_0)]\Delta c+\mathcal{O}(2)}{2k+\mathcal{O}(1)}
        \times \Bigg(\frac{\Big[\frac{2 k^3}{(k^2-4)^2}-\frac{2k \nu(r_0)}{k^2-4}\Big]\Delta c +\mathcal{O}(2)}{\frac{4}{k^2}+\mathcal{O}(1)}\Bigg)^{\frac{k^2+4}{k^2-4}},
        \label{expo}
    \end{aligned}\end{equation}

where $\mathcal{O}(1)$ denotes terms that are of at least
order $1$ in $\Delta c$ and $\varepsilon$, while $\mathcal{O}(2)$ stands for terms of at least order $2$ in $\Delta c$ and $r_0^{4/k^2}$. 
Solving for $\Delta c$ in \eqref{expo}, we obtain
\begin{equation}
    \begin{aligned}
        \Delta c=\alpha(k)\varepsilon^{1-\frac{4}{k^2}}[1+o(1)],
        \label{epspower}
    \end{aligned}
\end{equation}
where 
\begin{equation}
    \begin{aligned}
        \alpha(k)&=\frac{1}{r_0^{1-4/k^2}\big[k^2-(k^2-4)\nu(r_0)\big]}\frac{ (2k)^{1/2(1-4/k^2)}\big[2(k^2-4)^2\big]^{1/2(1+4/k^2)} }{k^{3/2(1+4/k^2)}}.
        \label{alphadef}
    \end{aligned}
\end{equation}
For future reference, we label the $r_0$-dependent contribution to $\alpha(k)$ as
\begin{equation}
    \begin{aligned}
        \delta(r_0)&=r_0^{1-4/k^2} \big[k^2-(k^2-4)\nu(r_0) \big]. \label{firstcoeff}
    \end{aligned}
\end{equation}
In spite of the function $\nu(r_0)$, as defined in Lemma~\ref{in/outlemma}, being dependent on $r_0$, that dependence must cancel, as the choice of $r_0$ in the definition of $\Sigma_1^{\rm in}$ is arbitrary. Therefore, we can take the limit as $r_0\to 0^+$ in \eqref{firstcoeff}.

\begin{lem}
    The function $\delta$ defined in Equation~\eqref{firstcoeff} satisfies 
    \begin{equation}
        \lim_{r_0\to0^+}\delta(r_0)=(k^2-4)\Gamma(1+4/k^2)\Gamma(1-4/k^2),
    \end{equation}
    where $k>2$.
    \label{limitlem}
\end{lem}

\begin{proof}
    We begin by writing $\delta(r_0)$ as 
    \begin{equation}
        \delta(r_0)=r_0^{1-4/k^2}k^2+\frac{k^2}{k^2+4}(k^2-4)(1-r_0) {_2}F_1\big(1+4/k^2, 4/k^2, 2+4/k^2, 1-r_0\big) ,
    \end{equation}
    using the definition of $\nu(r_0)$ from Lemma~\ref{in/outlemma}.
    Taking $r_0\to 0^+$, we find
    \begin{equation}
        \begin{aligned}
            \lim_{r_0\to0^+}\delta(r_0)&=\frac{k^2}{k^2+4} (k^2-4)\;{_2}F_1\big(1+4/k^2, 4/k^2, 2+4/k^2, 1\big)\\
            &=\frac{k^2-4}{1+4/k^2}\frac{\Gamma(2+4/k^2)\Gamma(1-4/k^2)}{\Gamma(1)\Gamma(2)}\\
            &=(k^2-4)\Gamma(1+4/k^2)\Gamma(1-4/k^2).
        \end{aligned}
    \end{equation}
    Here, we have used the identities ${_2}F_1(a,b,c,1)=\frac{\Gamma(c)\Gamma(c-a-b)}{\Gamma(c-a)\Gamma(c-b)}$~\cite[Equation 15.1.20]{abramowitz+stegun} and \\ $\Gamma(2+4/k^2)=(1+4/k^2)\Gamma(1+4/k^2)$, \textcolor{black}{as well as the fact that $\Gamma(1)=1=\Gamma(2)$}, which completes the proof.
\end{proof}

\begin{prop}
    Let $\varepsilon\in(0,\varepsilon_0)$, with $\varepsilon_0>0$ sufficiently small, and let $k>2$. Then, the function $\Delta c$ defined in Theorem~\ref{Myfirstthm} satisfies
    \begin{equation}
        \Delta c(\varepsilon)=\frac{2}{k^{1+8/k^2}}\frac{(k^2-4)^{4/k^2}}{\Gamma(1+4/k^2)\Gamma(1-4/k^2)}\varepsilon^{1-4/k^2}[1+o(1)].
        \label{Dc final}
    \end{equation}
    \label{Dc prop}
\end{prop}

\begin{proof}
    The statement follows directly from Lemma~\ref{limitlem} and Equations~\eqref{epspower} and \eqref{alphadef}.
\end{proof}

Hence, the proof of Theorem~\ref{Myfirstthm} is complete in the pushed front propagation regime, which is realised for $k>2$ in \eqref{main}.

\begin{remark}
We note that $\varepsilon^{1-4/k^2}\to \varepsilon^0=1$ as $k\to 2^+$ in \eqref{Dc final}, i.e., as we approach the pulled propagation regime. L'H\^opital's Rule shows that the corresponding coefficient tends to $0$ in that limit, which is consistent with Theorem~\ref{Myfirstthm}, as the correction $\Delta c$ is logarithmic in $\varepsilon$ for $k\leq 2$.
\end{remark}

\begin{remark}
    A simplification of the general expression for $\Delta c$ in Equation~\eqref{Dc final} is achieved for specific values of $k$ in \eqref{main}; e.g., $k=2\sqrt{2}$ gives $c(\varepsilon)=c(0)-\Delta c(\varepsilon)=\sqrt{2}+\frac{1}{\sqrt{2}}-\frac{1}{\pi}\varepsilon^{1/2}[1+o(1)]$. Similarly, for $k=4$, we have 
    $c(\varepsilon)=\frac{5}{2}-\frac{\sqrt[4]{3}}{\pi}\varepsilon^{3/4}[1+o(1)]$.
\end{remark}

\subsection{Numerical verification} 
\label{numerics section}
In this subsection, we verify the asymptotics in Theorem~\ref{Myfirstthm} by calculating numerically the error incurred by \\ ap\-prox\-i\-ma\-ting $c(\varepsilon)$ with the corresponding first-order expansion (in $\varepsilon$), which we denote by $\hat{c}(\varepsilon)$; for $k=4$, e.g., we have $\hat{c}(\varepsilon)=\frac{5}{2}-\frac{\sqrt[4]{3}}{\pi}\varepsilon^{3/4}$. The numerical value of $c(\varepsilon)$ is obtained by integrating Equation~\eqref{main} and storing the final value of $U=U_{\rm final}(c)$ obtained after a sufficiently large number of time steps. We then minimise $|U_{\rm final}(c)|$, taking $\hat{c}(\varepsilon)$ as our initial value of $c$. Our findings are illustrated in Figure~\ref{fig:k=4error} \textcolor{black}{for $k\in\{1,\frac{3}{2}, 2\sqrt{2}, 4$\}}, where we have used a double logarithmic scale, with $\varepsilon\in[10^{-4},10^{-2}]$.
\begin{figure}
    \includegraphics[scale = 0.75]{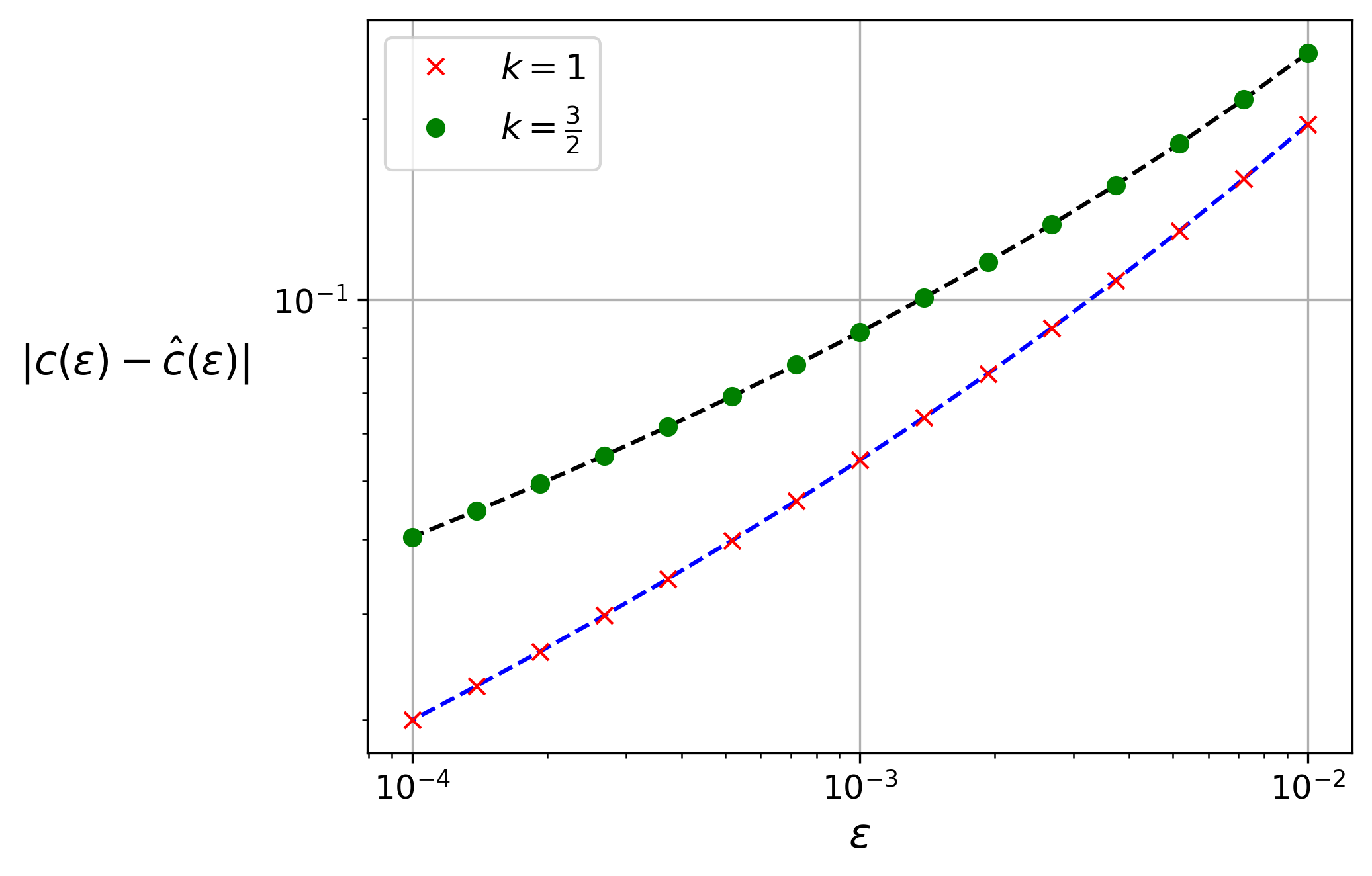}
    \includegraphics[scale=0.75]{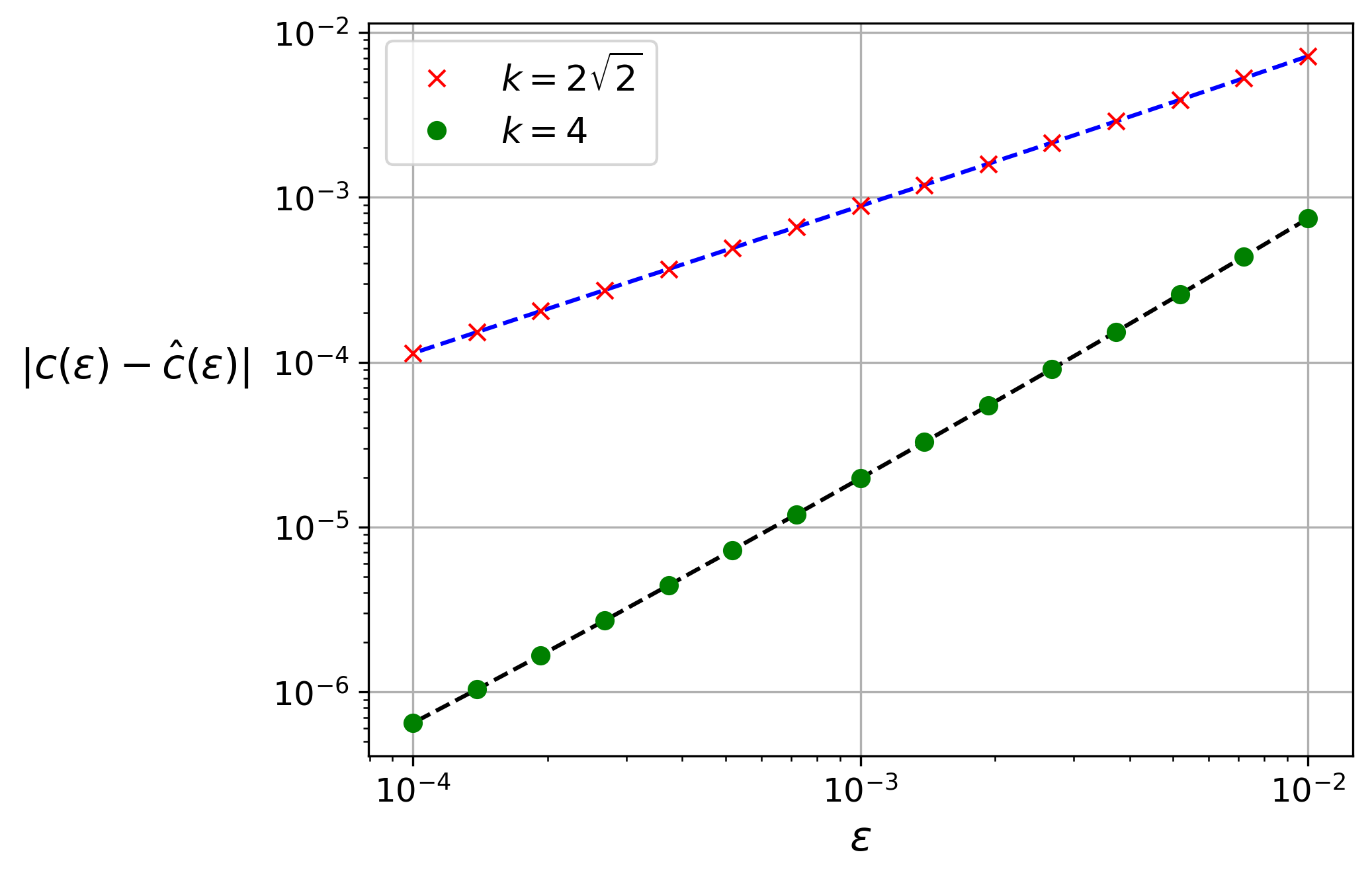}
    \caption{Error of the approximation of $c(\varepsilon)$ by $\hat{c}(\varepsilon)$ for $k\in\{1, \frac{3}{2}, 2\sqrt{2}, 4\}$ and $\varepsilon\in[10^{-4},10^{-2}]$.}
    \label{fig:k=4error}
\end{figure}
Figure~\ref{fig:k=4error} suggests that the next-order correction to $c(\varepsilon)$ will be of the order \textcolor{black}{$\mathcal{O}[(\ln{\varepsilon})^{-3}]$ for $k=1$ and $k=\frac{3}{2}$, whereas it will be $\mathcal{O}(\varepsilon)$ for $k=2\sqrt{2}$ and $\mathcal{O}(\varepsilon^{3/2})$ for $k=4$.}

\section{Discussion}
\label{discussionsection}
In this article, we have proven the existence of ``critical" travelling front solutions to the Burgers-FKPP equation with a Heaviside cut-off multiplying both the reaction kinetics and the advection term, recall Equation~\eqref{PDEcut1}. Moreover, we have rigorously derived the leading order $\varepsilon$-asymptotics of the unique front propagation speed $c(\varepsilon)$. \textcolor{black}{To the best of our knowledge, the effects of a cut-off on advection-reaction-diffusion equations of the type in \eqref{general 2} have not been studied before.}

For $k\leq 2$, the front is pulled and behaves as the pulled fronts with a cut-off considered, e.g., in~\cite{Dumortier_2007, Popović_2011}, with the correction to the front propagation speed being negative and of the order $\mathcal{O}[(\ln{\varepsilon})^{-2}]$. When $k>2$, the front is pushed, and the correction to the speed of propagation is also negative, and proportional to a fractional power of $\varepsilon$, again in analogy to the pushed fronts in reaction-diffusion equations with a cut-off analysed in~\cite{Popović_2011}.
\textcolor{black}{While the proof of Theorem~\ref{Myfirstthm} in the pulled propagation regime where $k\leq 2$ closely follows the proof of \cite[Theorem~1]{Dumortier_2007}, we have included it for completeness. The analysis of the pushed regime, with $k>2$, is significantly more involved algebraically and relies on a modification of the approach developed in \cite{popdeg, Popović_2011, POPOVIC20121976}. Our main analytical contribution in this article can be found in Section~\ref{persistencesection}, where we adapt techniques from both \cite{Dumortier_2007} and \cite{popdeg, Popović_2011, POPOVIC20121976} to derive the asymptotics in \eqref{dc in thm}.}

\textcolor{black}{It is important to emphasise that the blow-up technique is applied in the present context to remedy a discontinuity in the governing equations, rather than a loss of hyperbolicity, as is typically the case in applications of blow-up \cite{krupa33extending,krupa2001blowup}. Given that the regularisation of piecewise smooth systems via the alternative methodology developed in \cite{zbMATH01531778} typically results in a singular perturbation problem, it may be feasible to adapt that well-established methodology to our setting; see \cite{zbMATH06560497} for a specific application.}

It may be possible to calculate higher-order terms in $\varepsilon$ in the expansion of $c(\varepsilon)$ in the pushed regime; however, to do so, one must solve for $\frac{\partial^j V}{\partial c^j}(U,c(0))$ ($j\geq 2$) via the procedure outlined in Lemma~\ref{variational_lemma}.
We note that the procedure will fail for general pulled fronts with $k<2$, as the front is not explicitly known in those cases, preventing us from calculating $W^{\rm in}$ to a sufficiently high order to determine higher-order terms in $c(\varepsilon)$.
However, at the boundary between the pushed and pulled regimes, when $k=2$, the front is known explicitly, as is $\frac{\partial V}{\partial c}(U,c(0))$, see Lemma~\ref{variational_lemma}. Hence, one could approximate $W^{\rm in}$ to a sufficiently high order, which may make it possible to determine the next term in the expansion of $c(\varepsilon)$ when $k=2$.

By retracing the proof of Theorem~\ref{Myfirstthm}, one can show that the leading-order correction to the front speed is independent of whether or not the advection term in Equation~\eqref{PDEcut1} is multiplied with the cut-off function $H(u-\varepsilon)$. That observation is supported by the motivating example of the Burgers equation with a cut-off, Equation~\eqref{simpleburger with cut}, where the correction to the front speed is given by $\Delta c(\varepsilon)=\frac{k}{2}\varepsilon^2$ to leading order. As both $(\ln{\varepsilon})^{-2}$ and $\varepsilon^{1-4/k^2}$ are of lower order compared to $\varepsilon^2$, it is to be expected that a cut-off in the advection term will not affect the leading-order asymptotics of $\Delta c(\varepsilon)$. 
Since the requisite argument is very similar to the proof of Theorem~\ref{Myfirstthm}, we outline it in Appendix~\ref{appendix}. 

\textcolor{black}{In Theorem~\ref{Myfirstthm}, we have restricted to a Heaviside cut-off function $H(u-\varepsilon)$ in Equation~\eqref{PDEcut1}; that restriction appears reasonable, as the Heaviside cut-off cancels the reaction kinetics and the advection term exactly when no particles are present in the underlying $N$-particle system. One can instead introduce a general cut-off function $\psi(u, \varepsilon)$ in \eqref{PDEcut1}} which satisfies $\psi(u,\varepsilon)\equiv 1$ when $u>\varepsilon$ and $\psi(u,\varepsilon)<1$ for $u<\varepsilon$. Equation~\eqref{general} with Fisher reaction kinetics and a general cut-off has been considered in \cite{Dumortier_2007}, while a linear cut-off function was studied explicitly in \cite{POPOVIC20121976}. In the context of Equation~\eqref{PDEcut1}, one can show that to leading order, $\Delta c=\frac{\pi^2}{\ln(\varepsilon)^2}$ in the pulled propagation regime \textcolor{black}{for a wide range of cut-off functions which includes the Heaviside cut-off \cite{Dumortier_2007}}. Hence, the leading-order asymptotics of $\Delta c$ is then universal, as was also the case in \cite{Dumortier_2007}. In the pushed regime, one again obtains $\Delta c=\mathcal{O}\big(\varepsilon^{1-4/k^2}\big)$; however, it is not possible to calculate explicitly the corresponding leading-order coefficient for a general cut-off function $\psi$, since explicit knowledge of the entry point in chart $K_2$ is required. \textcolor{black}{In particular, that coefficient will be cut-off-dependent then, in contrast to the pulled propagation regime, as is also the case in \cite{Dumortier_2007, Popović_2011, PhysRevE.72.056113}}.

Finally, we note that Theorem~\ref{Myfirstthm} can be extended to advection-reaction-diffusion equations with a more general advection term,
\begin{equation}
    \frac{\partial u}{\partial t}+ku^n\frac{\partial u}{\partial x} H(u-\varepsilon)=\frac{\partial^2 u}{\partial x^2}+u(1-u^n)H(u-\varepsilon),
    \label{PDEcut-2}
\end{equation}
where $n\geq 2$ is integer-valued. For $\varepsilon=0$, Equation~\eqref{PDEcut-2} admits a pulled front for $k\leq n+1$ and $c\geq 2$, whereas for $k>n+1,$ there exists a pushed front solution for $c\geq\frac{k}{n+1}+\frac{n+1}{k}$, which can be shown in analogy to the proof of Theorem~\ref{basic thm} via the approach outlined in \cite{2021ZaMP...72..163M}. For $c=\frac{k}{n+1}+\frac{n+1}{k},$ the sought-after front corresponds, in a co-moving frame, to the heteroclinic orbit $V(U)=-\frac{k}{n+1}U(1-U^n)$. Due to the increased algebraic complexity, we leave the study of the impact of a cut-off on Equation~\eqref{PDEcut-2} for the future.
However, we note that, as the front is explicitly known in the pushed regime, it is likely that the leading-order correction to the propagation speed $c(\varepsilon)$ can be calculated explicitly for all $k>0$.

\appendix
\section{Proof of Theorem~\ref{Myfirstthm} without cut-off in advection} 
\label{appendix}

Here, we briefly show that Theorem~\ref{Myfirstthm} remains equally valid for the advection-reaction-diffusion-equation 
\begin{equation}
    \frac{\partial u}{\partial t}+ku\frac{\partial u}{\partial x}=\frac{\partial^2 u}{\partial x^2}+u(1-u)H(u-\varepsilon),
    \label{Burger-without-Acut}
\end{equation}
in which the advection term $ku\frac{\partial u}{\partial x}$ is not affected by the cut-off. The corresponding first-order system then reads
\begin{equation}
    \begin{aligned}
        U'&=V,\\
        V'&=-\gamma V+kUV-U(1-U)H(U-\varepsilon),\\
        \varepsilon'&=0,
        \label{firstordersystem2}
    \end{aligned}
\end{equation}
where the travelling wave variable is now defined by $\xi=x-\gamma t$, with $\gamma$ the front propagation speed. 
The analysis of \eqref{firstordersystem2} again relies on the blow-up transformation in \eqref{burger-blowup1}. We observe that \eqref{PDEcut1} and \eqref{Burger-without-Acut} are identical for $u>\varepsilon$; therefore, it suffices to study \eqref{firstordersystem2} in the rescaling chart $K_2$ only, which is again defined by \eqref{K_2_blowup}:
\begin{equation}
    \begin{aligned}
        u_2'&=v_2, \\
        v_2'&=-\gamma v_2+kr_2u_2v_2,\\
        r_2'&=0.
        \label{noAcutfirstsystem}
    \end{aligned}
\end{equation}
By taking $\varepsilon=r_2\to 0^+$ and solving for $v_2(u_2)$, we obtain \eqref{Gamma2}, i.e., the singular orbit $\Gamma_2$ in $K_2$ is given as before. By the above, we can conclude that Proposition~\ref{propconst} holds for \eqref{Burger-without-Acut}. 

It remains to consider the persistence of $\Gamma$ for \eqref{Burger-without-Acut}.
We can solve \eqref{noAcutfirstsystem} explicitly for general $\gamma$ and $\varepsilon(=r_2)>0$ to obtain 
\begin{equation}
   W_2^{\rm s}(\ell_2^+):\ v_2(u_2)=-\gamma u_2+\frac{k}{2}r_2 u_2^2.
   \label{stablemanifold}
\end{equation}
The point of intersection of $W_2^{\rm s}(\ell_2^+)$ with $\Sigma_2^{\rm in}$ is given by $P_2^{\rm in}=(1, v_2^{\rm in}, \varepsilon)$, where $v_2^{\rm in}=-\gamma+\frac{k}{2}\varepsilon.$ From the definition of \eqref{burger-blowup1}, it follows that $V^{\rm in}=v_2^{\rm in}\varepsilon =-\gamma\varepsilon+\frac{k}{2}\varepsilon^2$.
One can now prove a similar result to Proposition~\ref{persistenceprop} for \eqref{firstordersystem2}, where $c(0)=c_{\rm crit}$ is again defined as in Theorem~\ref{basic thm}.

\begin{prop}
    For $\varepsilon\in(0, \varepsilon_0)$, with $\varepsilon_0$ sufficiently small, $k>0$, and $\gamma$ close to
    $c(0)$, there exists a critical heteroclinic connection between $Q^-$ and $Q^+$ in Equation~\eqref{firstordersystem2} for a unique speed $\gamma(\varepsilon)$ which depends on $k$. Furthermore, there holds $\gamma(\varepsilon)\leq c(0)$.
    \label{persistenceprop2}
\end{prop}

The proof is similar to that of Proposition~\ref{persistenceprop}, the only difference being that $V^{\rm in}=-c\varepsilon$ is replaced by $V^{\rm in}=v_2^{\rm in}\varepsilon=-\gamma\varepsilon+\frac{k}{2}\varepsilon^2$. In spite of that difference, the argument from the proof of Proposition~\ref{persistenceprop} carries over verbatim.

The remainder of the analysis in Section~\ref{persistencesection} equally translates to Equation~\eqref{Burger-without-Acut}. The sole difference concerns the point $P_1^{\rm out}=(\varepsilon, W^{\rm out}, 1)$, where $W^{\rm out}$ is derived from $v_2^{\rm in}$. We have the following result on the leading-order asymptotics of $W^{\rm out}$.
\begin{lem}
    For $k>2$ and $\varepsilon$ and $\Delta \gamma$ sufficiently small, the point $P_1^{\rm out}=(\varepsilon,W^{\rm out},1)$ satisfies \begin{equation}
        W^{\rm out}=-\frac{2}{k} \textcolor{black}{+\mathcal{O}(\Delta \gamma, \varepsilon),}
        \label{W^{out}new}
    \end{equation} 
    where $\gamma(\varepsilon)=c(0)-\Delta \gamma.$
    \label{lem new W^{in}}
\end{lem}
\begin{proof}
   Equation~\eqref{W^{out}new} follows from the definition of $V^{\rm in}$ for \eqref{firstordersystem2} and the sequence of transformations defined in Lemma~\ref{normal_form_1}.
\end{proof} 

In Section~\ref{dc pulled}, i.e., in the pulled regime, we found that the leading-order asymptotics of $\Delta c$ is independent of $W^{\rm in}$ and $W^{\rm out}$; therefore, we can conclude that Theorem~\ref{Myfirstthm} holds for Equation~\eqref{Burger-without-Acut} when $k\leq 2$, i.e., that $\Delta \gamma =\Delta c$ to leading order.

As the asymptotics of $W^{\rm out}$ for \eqref{Burger-without-Acut} in Lemma~\ref{lem new W^{in}} differs from that in Lemma~\ref{in/outlemma} at $\mathcal{O}(\varepsilon)$, and as only the constant term $-\frac{2}{k}$ is required to derive the leading-order asymptotics of $\Delta c$, we can conclude that Theorem~\ref{Myfirstthm} holds for \eqref{Burger-without-Acut} in the pushed front propagation regime where $k>2$.

In Figure~\ref{fig:k=4appendixplot}, we compare the propagation speeds $c(\varepsilon)$ and $\gamma(\varepsilon)$ for \eqref{PDEcut1} and \eqref{PDEcut-2}, respectively. We find that $|c(\varepsilon)-\gamma(\varepsilon)|$ is of higher order than $|c(\varepsilon)-\hat{c}(\varepsilon)|$, which we plot in red \textcolor{black}{for comparison}. For example, for $k=4$, the propagation speeds $c(\varepsilon)$ and $\gamma(\varepsilon)$ differ approximately at order $\mathcal{O}(\varepsilon^{9/5})$, whereas the difference between the propagation speed $c(\varepsilon)$ for \eqref{PDEcut1} and its leading-order approximation $\hat{c}(\varepsilon),$ given in Theorem~\ref{Myfirstthm}, is of the order $\mathcal{O}(\varepsilon^{3/2})$.

\begin{figure}[H]
    \centering
    \includegraphics[scale = 0.75]{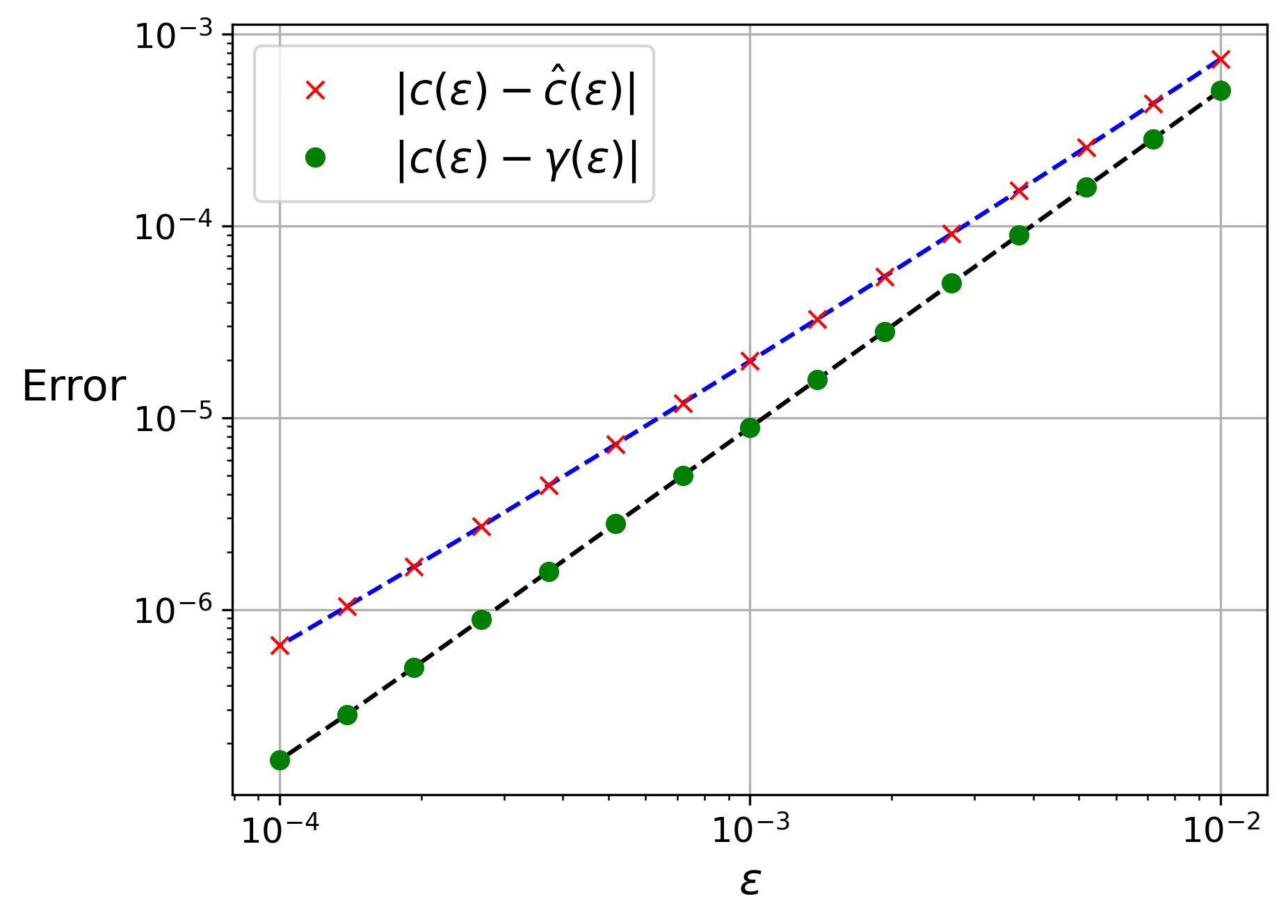}
    \caption{Numerical difference between $c(\varepsilon)$ and $\gamma(\varepsilon)$ (green) for $k=4$ and $\varepsilon\in[10^{-4},10^{-2}]$; the error $|c(\varepsilon)-\hat{c}(\varepsilon)|$ is plotted for comparison (red).}
    \label{fig:k=4appendixplot}
\end{figure}

\section*{Acknowledgements}
ZS was supported by the EPSRC Centre for Doctoral Training
in Mathematical Modelling, Analysis and Computation (MAC-MIGS) funded
by the UK Engineering and Physical Sciences Research Council (EPSRC)
grant EP/S023291/1, Heriot-Watt University, and the University of Edinburgh. \textcolor{black}{The authors are grateful to Qixiang Xu
(School of Mathematical Sciences, Beijing Normal University) for alerting them to a mistake in the original proof of Lemma~\ref{lem normal_form}.}

\addcontentsline{toc}{section}{References}
\printbibliography

\end{document}